\documentclass[mathpazo]{cicp} 
\usepackage{bm,bbold,subfig,booktabs,cases}
\usepackage{algorithm,algorithmic}
\usepackage[title]{appendix}
\graphicspath{{figure/}}

\DeclareMathOperator{\sgn}{sgn}
\DeclareMathOperator{\minmod}{minmod}
\DeclareMathOperator{\Span}{Span}
\DeclareMathOperator{\st}{s.t.}

\begin{document}
\title{Integrated Linear Reconstruction for Finite Volume Scheme on
        Arbitrary Unstructured Grids}

\author{
  Li Chen\affil{1}, 
  Guanghui Hu\affil{2,3}\corrauth, 
  and Ruo Li\affil{4}
}

\address{
  \affilnum{1}\ School of Mathematical Sciences, Peking University, 
                      Beijing, China\\
  \affilnum{2}\ Department of Mathematics, University of Macau, 
                     Macao SAR, China\\
  \affilnum{3}\ UM Zhuhai Research Institute, Zhuhai, Guangdong, China\\
  \affilnum{4}\ HEDPS \& CAPT, LMAM \& School of Mathematical 
                    Sciences, Peking University, Beijing, China
}

\emails{
  {\tt cheney@pku.edu.cn} (L.~Chen), 
  {\tt garyhu@umac.mo} (G.H.~Hu), 
  {\tt rli@math.pku.edu.cn} (R.~Li)
}

\begin{abstract}
  In [L. Chen and R. Li, Journal of Scientific Computing, Vol. 68, pp. 
  1172--1197, (2016)], an integrated linear reconstruction was proposed for 
  finite volume methods on unstructured grids.  However, the geometric 
  hypothesis of the mesh to enforce a local maximum principle is too 
  restrictive to be satisfied by, for example, locally refined meshes or 
  distorted meshes generated by arbitrary Lagrangian-Eulerian methods in 
  practical applications.  In this paper, we propose an improved integrated 
  linear reconstruction approach to get rid of the geometric hypothesis. The 
  resulting optimization problem is a convex quadratic programming problem, 
  and hence can be solved efficiently by classical active-set methods.  The 
  features of the improved integrated linear reconstruction include that i). the 
  local maximum principle is fulfilled on arbitrary unstructured grids, ii). the
  reconstruction is parameter-free, and iii). the finite volume scheme is 
  positivity-preserving when the reconstruction is generalized to the Euler 
  equations.  A variety of numerical experiments are presented to 
  demonstrate the performance of this method.
\end{abstract}

\maketitle

\section{Introduction}

The high-order finite volume schemes can be summarized by a
reconstruct-evolve-average (REA) process, i.e. a piecewise polynomial is
reconstructed in each cell with given cell averages, then the governing 
equation is evolved according to those polynomials, and finally the cell 
averages are recalculated.  Among these three stages, reconstruction plays 
an important role in giving high-order solutions without numerical 
oscillations.  Nowadays, second-order methods have been the workhorse for 
computational fluid dynamics \cite{Wang2013}.  To achieve second-order 
accuracy, a prediction of the gradient is obtained first.  Due to the possible 
non-physical flow caused by the underestimation or overestimation of the 
gradient, it is necessary to limit the gradient in a proper way, and thus the 
prediction-limiting algorithm arises.  In one-dimensional case, total variation 
diminishing (TVD) limiters are commonly used in designing high-resolution 
schemes for conservation laws.  Unfortunately, it is very difficult to 
implement the TVD limiter for multi-dimensional problems, especially on 
unstructured grids.  To get around this negative result, a new class of 
positive schemes has been proposed \cite{Spekreijse1987} which
ensures a local maximum principle.  Since the introduction of this idea, a 
large number of limiters have then been developed.  These limiters include, 
among others, the Barth's limiter \cite{Barth1989}, the Liu's limiter 
\cite{Liu1993}, the maximum limited gradient (MLG) limiter 
\cite{Batten1996} and the projected limited central difference (PLCD) limiter 
\cite{Hubbard1999}.  See \cite{Hubbard1999} for a comprehensive 
comparison of these limiters.  More recently, Park \textit{et al.} successfully 
extended the multi-dimensional limiting process (MLP) introduced in 
\cite{Kim2005} from structured grids to unstructured grids 
\cite{Park2010}.  Li \textit{et al.} proposed the weighted biased averaging 
procedure (WBAP) limiter \cite{Li2011} based upon the biased averaging 
procedure (BAP) limiter which was introduced in \cite{Choi1998}.
Towards the positivity-preserving property, which is crucial for the stability
on solving the Euler equations, there have also been several pioneering 
works. For instance, Perthame and Shu \cite{Perthame1996} developed a 
general finite volume framework on preserving the positivity of density and 
pressure when solving the Euler equations. Motivated by this work, the 
framework was extended to the discontinuous Galerkin method on 
rectangular meshes \cite{Zhang2010} and on triangular meshes 
\cite{Zhang2012}, and to the Runge-Kutta discontinuous Galerkin method 
\cite{Liu2016}. More recently, a parametrized limiting technique was 
proposed in \cite{Christlieb2015} to preserve the positivity property on
solving the Euler equations on unstructured grids.

Besides the above prediction-limiting algorithm in the reconstruction, there 
are also methods which deliver the limited gradient in a single process.  For 
example, Chen and Li  \cite{Chen2016} introduced the concept of 
integrated linear reconstruction (ILR), in which the limited gradient is 
computed by solving a linear programming on each cell using an efficient 
iterative method.  A similar approach was given by May and Berger 
\cite{May2013}.  However, the fulfillment of local maximum principle of 
these methods requires certain geometric hypothesis on the grids 
\cite{Chen2016}.  Buffard and Clain \cite{Buffard2010} considered a 
monoslope MUSCL method, where a least-squares problem subjected to 
maximum principle constraints was imposed.  They solved the optimization 
problem explicitly.  However, the involved cases for triangular grids, 
discussed in the article, were rather complicated, let alone irregular grids 
with hanging nodes appeared in the computational practice such as mesh 
adaptation.  This motivates us to discard the unsatisfactory geometric
hypothesis on unstructured grids by imposing constraints on the quadrature
points, and to solve the optimization problem iteratively.  In our linear
reconstruction, the gradient is indeed obtained by solving a quadratic
programming problem using the active-set method.  It can be verified that 
the resulting finite volume scheme for scalar conservation laws satisfies a 
local maximum principle on arbitrary unstructured grids.  Besides, this linear
reconstruction can be easily adapted to the Euler equations.  And it can be 
shown that the numerical solutions preserve the positivity of density and 
pressure.

The remainder of this paper is organized as follows.  In Section 2, we review
the MUSCL-type finite volume scheme on unstructured grids.  In Section 3, 
we describe the improved integrated linear reconstruction based on solving 
a series of quadratic programming problems.  Section 4 is devoted to the 
discussion of local maximum principle for scalar conservation laws as well 
as the positivity-preserving property for the Euler equations.  Numerical 
results are shown for the scalar problems and the Euler equations to 
validate the effectiveness and robustness of our method in Section 5.  
Finally, a short conclusion will be drawn in Section 6.

\section{MUSCL-type finite volume scheme}

In this section we briefly introduce the MUSCL-type finite volume methods 
for hyperbolic conservation laws on two-dimensional unstructured grids. 
The extension to three-dimensional cases is straightforward.
Consider the following system of hyperbolic conservation laws:
\begin{equation}
  \dfrac{\partial\bm u}{\partial t}+\nabla\cdot\bm F(\bm u)=\bm 0.
  \label{hcl}
\end{equation}
The computational domain is triangulated into an unstructured grid 
$\mathcal{T}$.  To obtain the MUSCL-type finite volume scheme,
we integrate \eqref{hcl} on some control volume $T_0\in\mathcal T$ 
\begin{equation}
  \dfrac{\mathrm d\bm u_0}{\mathrm dt}+\dfrac1{|T_0|}\oint_{\partial T_0}\bm{F}(\bm u) 
  \cdot \bm{n} \mathrm ds=\bm 0,
  \label{semidiscrete}
\end{equation}
where the cell-averaged solution $\bm u_0(t)=\dfrac{1}{|T_0|}
\displaystyle\int _{T_0}\bm u(\bm x,t)\mathrm d\bm x$ 
and $\bm n$ is the unit outward normal over the boundary $\partial T_0$
consisting of $J$ edges $e_1,e_2,\cdots,e_J$.

Since we are concerned with second-order schemes, the boundary integral 
appeared in \eqref{semidiscrete} is approximated by the midpoint 
quadrature rule, namely,
\begin{equation}
\oint_{\partial T_0}\bm F(\bm u)\cdot\bm n\mathrm ds 
\approx \sum_{j=1}^J |e_j| \bm F(\bm u(\bm z_{j}, t))\cdot\bm n_{j},
\label{GaussQuad}
\end{equation}
where $\bm z_j$ and $\bm n_j$ represent the midpoint and normal of the 
edge $e_j~(j=1,2,\cdots,J)$. To maintain stability, the flux term 
$\bm F(\bm u(\bm z_j,t))\cdot\bm n_{j}$ appeared in \eqref{GaussQuad}
is replaced by a numerical flux function 
$\mathcal F(\bm u_{j}^-,\bm u_{j}^+;\bm n_{j})$,
where $\bm u_{j}^\pm$ denote values of numerical solutions at $\bm z_j$ 
outside/inside the cell $T_0$ respectively $(j=1,2,\cdots,J)$. The resulting 
semi-discrete scheme is
\[
\dfrac{\mathrm d\bm u_0}{\mathrm dt} + \dfrac1{|T_0|} \sum_{j=1}^J
\mathcal F(\bm u_{j}^-,\bm u_{j}^+;\bm n_{j}) |e_{j}| = \bm 0.
\]

To achieve second-order numerical accuracy, the values $\bm u_{j}^\pm$ 
are evaluated from the piecewise linear solution reconstructed from nearby 
cell averages. In Section 3 we will describe how to construct an 
appropriate piecewise linear solution which satisfies certain properties such 
as local maximum principle or positivity-preserving property. After that, we 
obtain a system of ordinary differential equations for $\bm u_h=
\displaystyle\sum_{T_0\in\mathcal T}\bm u_0\mathbb 1_{T_0}$:
\begin{equation}
\dfrac{\mathrm d\bm u_h}{\mathrm dt}=\mathcal L(\bm u_h)
= -\sum_{T_0\in\mathcal T}\sum_{j=1}^J\dfrac{|e_j|}{|T_0|}
\mathcal F(\bm u_j^-,\bm u_j^+;\bm n_j)\mathbb 1_{T_0}.
\label{eq:ode}
\end{equation}

Finally, the system \eqref{eq:ode} is discretized by the SSP Runge-Kutta 
method to achieve second-order temporal accuracy \cite{Gottlieb2001}:
\begin{equation}
  \begin{cases}
    \bm u^*_h  = \bm u^{n}_h+\Delta t_n {\mathcal L}(\bm u^{n}_h), \\
    \bm u^{n+1}_h  = \dfrac12\bm u^{n}_h+\dfrac12 \left(
      \bm u^{*}_h+\Delta t_n {\mathcal L}(\bm u^{*}_h)\right),
  \end{cases}
  \label{eq:TVDRK2}
\end{equation}
where the time step length $\Delta t_n$ is determined through the CFL 
condition. 

To effectively control the total amount of cells,  we adopt the $h$-adaptive 
method following Li \cite{Li2005}. This method relaxes the regularity 
requirement of the mesh. Therefore, even if the background mesh is a 
regular mesh where any two adjacent cells share a whole edge, an irregular 
mesh may be produced during the refinement process. The irregular 
structure may lead to different resolution of the adjacent cells used for 
computing numerical fluxes. Fortunately, the strategy of mesh refinement 
proposed in \cite{Li2005} prevents the adjacent cells to differ more than 
one level of resolution, and thus there are essentially three possibilities 
which are depicted schematically in Fig. \ref{fig:fluxcom}. The shaded 
triangle appeared in case b) is called \emph{twin triangle}, since it can be 
divided into two smaller triangles by one of its medians. The twin triangle 
has one special edge shared by two neighbors. However, if we treat the twin 
triangle as a quadrilateral with $J=4$ edges, then all the computations can 
be carried out in the usual way.

\begin{figure}[htbp]
  \centering
  \subfloat[]{\includegraphics[width=.3\textwidth]{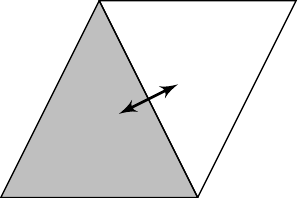}}
  \quad
  \subfloat[]{\includegraphics[width=.3\textwidth]{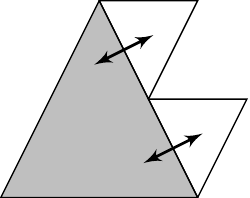}}
  \quad
  \subfloat[]{\includegraphics[width=.3\textwidth]{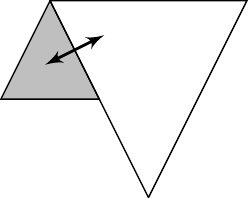}}
    \caption{Flux computation for cells with
     a) a neighbor of the same resolution;
     b) neighbors of higher resolution;
     c) a neighbor of lower resolution.}
    \label{fig:fluxcom}
\end{figure}

\section{Construction of the integrated linear reconstruction}

In this section we describe the improved integrated linear reconstruction. 
Consider the solution $u(\bm x,t)$ of a scalar conservation law. For 
simplicity we will omit the dependency in time throughout this discussion.
The reconstructed linear function on a given cell $T_0\in\mathcal T$ with
the centroid $\bm x_0$ can be formulated as 
\begin{equation}
\hat u_0(\bm x)=u_0+\bm L^\top(\bm x-\bm x_0),
\label{eq:linearfun}
\end{equation}
where $\bm L$ is the gradient vector.

To compute the gradient $\bm L$ we require information from neighbors of 
$T_0$. In the interior of the computational domain the von Neumann 
neighbors are sufficient for reconstruction. Nevertheless, this is not the 
case on the boundary, where we need special treatment, as will be 
described in Section \ref{sec:bndtr}.

To achieve better accuracy without non-physical oscillation, we minimize the 
residuals over all gradient $\bm L$ subjected to some stability conditions.
The solution of this optimization problem, if exists, is used to construct a 
piecewise linear solution. This is the rough idea of the integrated linear 
reconstruction. 

\subsection{Formulation of optimization problem}

In the integrated linear reconstruction proposed in this paper, the objective 
function is the sum of \emph{squared} residuals, rather than the sum of 
\emph{absolute} residuals that was used in \cite{Chen2016},
\begin{equation}
\delta (\bm L)=\sum_{j=1}^J (\hat u_0(\bm x_j)-u_j)^2.
\label{eq:objfun}
\end{equation}
Potentially, the smoothness of the reconstructed solutions can be improved. 
In \cite{Chen2016}, it has been shown that to theoretically guarantee the 
maximum principle, a certain hypothesis on the geometry of elements 
needs to be satisfied. The hypothesis can be smoothly satisfied with quality 
Delaunay triangular mesh, which can be generated by most mature tools 
such as EasyMesh \cite{EasyMesh}. However, it can be shown that the 
hypothesis can be easily violated on distorted meshes. To resolve this 
issue, we discard the constraints on the neighboring centroids and impose 
the following inequalities on the midpoint $\bm z_j$:
\begin{equation}
m_j\le \hat u_0(\bm z_{j})\le M_j,\quad j=1,2,\cdots,J,
\label{eq:constraints}
\end{equation}
where the lower and upper bounds are given by
\[
m_j=\min\{u_0,u_j\},\quad 
M_j=\max\{u_0,u_j\},\quad 
j=1,2,\cdots,J.
\]
Now using the definition \eqref{eq:linearfun} of linear function 
$\hat u_0(\bm x)$,  the objective function \eqref{eq:objfun} becomes
\[
\begin{split}
\delta(\bm L)&=\sum_{j=1}^J (\hat u_0(\bm x_j)-u_j)^2=\sum_{j=1}^J 
(u_0-u_j+\bm r_j^\top\bm L)^2 \\
& =\sum_{j=1}^J \left((u_0-u_j)^2+2(u_0-u_j)
\bm r_j^\top\bm L+\bm L^\top\bm r_j\bm r_j^\top\bm L\right)\\
& = 
\bm L^\top\mathbf G\bm L+2\bm c^\top\bm L+\text{const,}
\end{split}
\]
and the constraints \eqref{eq:constraints} become
\[
m_j\le \hat u_0(\bm z_{j})=u_0+(\bm z_{j}-\bm x_0)^\top\bm L
=u_0+\bm a_{j}^\top\bm L\le M_j,
\]
where 
\[
\mathbf G=\sum_{j=1}^J \bm r_j\bm r_j^\top,\quad
\bm c= \sum_{j=1}^J (u_0-u_j)\bm r_j,\quad
\bm r_j = \bm x_j-\bm x_0,\quad
\bm a_j = \bm z_j-\bm x_0,\quad
j=1,2,\cdots,J.
\]
Consequently, the linear reconstruction reduces to the following 
\emph{double-inequality constrained quadratic programming (QP) problem}
\begin{equation}
\begin{aligned}
\min~ & \dfrac{1}{2}\bm L^\top\mathbf G\bm L+\bm c^\top\bm L \\
\st~ & \overline{\bm m}\le \mathbf A\bm L\le \overline{\bm M},
\end{aligned}
\label{eq:dic}
\end{equation}
where 
\begin{gather*}
\mathbf A=\left[\bm a_1,\bm a_2,\cdots,\bm a_J\right]^\top,\quad
\overline{\bm m}=
\left[\overline m_1,\overline m_2,\cdots,\overline m_J\right]^\top,\quad 
\overline{\bm M}=
\left[\overline M_1,\overline M_2,\cdots,\overline M_J\right]^\top,\\
\overline m_j =m_j-u_0,\quad
\overline M_j =M_j-u_0,\quad j=1,2,\cdots,J.
\end{gather*}

It is reasonable to assume that 
\[
\Span\{\bm r_1,\bm r_2,\cdots,\bm r_J\}=\mathbb R^2,
\]
which is true for most of the triangular grids in practical computations.
Then the matrix $\mathbf G$ is positive definite and thus the existence and 
uniqueness of the minimizer $\bm L$ are guaranteed. With the gradient 
$\bm L$ in hand, we build the interior value $u_j^-$ as follows:
\[
u_j^-=\hat u_0(\bm z_j)=u_0+\bm L^\top(\bm z_j-\bm x_0),\quad
j=1,2,\cdots,J.
\]
In the following discussion we develop an efficient solver for the problem 
\eqref{eq:dic}.

\subsection{An active-set method for quadratic programming problem}

The active-set method has been widely used since the 1970s and is 
effective for small- and medium-sized QP problems. This method updates 
the solution by solving a series of QP subproblems in which some of the 
inequality constraints are imposed as equalities. As for the 
double-inequality constrained problem \eqref{eq:dic}, any constraint is 
assumed to be active on a single side on each step in the active-set 
method, even if both sides of this constraint hold with equality. Therefore, 
we introduce a balanced ternary variable $\delta_j$ to indicate the active 
state of $j$-th constraint, i.e.
\[
\delta_j=
\begin{cases}
+1, & \text{if left-hand side of } j \text{-th constraint is active},\\
-1, & \text{if right-hand side of } j \text{-th constraint is active},\\
0, & \text{if neither side of } j \text{-th constraint is active}.
\end{cases}
\]
We generate iterates of \eqref{eq:dic} that remain feasible while steadily 
decreasing the objective function. To be more precise, given $\bm L_k$ 
obtained from the $k$-th iteration, we solve the following QP problem for 
the descending direction $\bm p_k$,
\begin{equation}
\begin{split}
\min~ & \dfrac{1}{2}\bm p_k^\top\mathbf G\bm p_k+
(\mathbf G\bm L_k+\bm c)^\top\bm p_k \\
\st~ & \mathbf M \bm p_k=\bm 0,
\end{split}
\label{eq:subqp}
\end{equation}
where $\mathbf M=[\delta_j\bm a_j^\top]_{\delta_j\neq 0}$ is composed 
of normals of the active constraints. 

The first-order necessary conditions for $\bm p_k$ to be a solution of 
\eqref{eq:subqp}
imply that there is a vector of Lagrange multipliers $\bm{\lambda}$ such 
that 
\begin{subnumcases}{}
\mathbf M^\top\bm\lambda-\mathbf G\bm p_k=\mathbf G\bm{L_k}+
\bm c , 
\label{eq:KKT1}
\\
\mathbf M\bm p_k=\bm 0  .
\label{eq:KKT2}
\end{subnumcases}
Multiply \eqref{eq:KKT1} by $\mathbf M\mathbf G^{-1}$ and then add 
\eqref{eq:KKT2} to obtain a linear system in the vector $\bm{\lambda}$
alone:
\[
(\mathbf M\mathbf G^{-1}\mathbf M^\top)\bm\lambda=
\mathbf M(\bm L_k+\mathbf G^{-1}\bm c).
\]
We solve this symmetric positive definite system for $\bm{\lambda}$. Then 
$\bm p_k$ can be recovered from 
\[
\mathbf G\bm p_k=
\mathbf M^\top\bm\lambda-\mathbf G\bm L_k-\bm c.
\]

If the descending direction $\bm p_k$ is nonzero, we set 
$\bm L_{k+1}=\bm L_k+\alpha_k\bm p_k$ where the step-length 
parameter $\alpha_k$ is chosen to be the largest value in the range 
$[0, 1]$ such that all constraints are satisfied. Indeed an explicit formula for 
$\alpha_k$ can be derived. To satisfy the $j$-th constraint, we have
\[
m_j \le u_0+\bm a_j^\top(\bm L_k+\alpha_k\bm p_k)\le M_j,
\]
or
\[
m_j-u_0-\bm a_j^\top\bm L_k \le (\bm a_j^\top\bm p_k)\alpha_k
\le M_j-u_0-\bm a_j^\top\bm L_k.
\]
Therefore, to maximize the decrement of objective function, the step-length 
parameter $\alpha_k$ can be chosen in the following way:
\[
\alpha_k:=\min\left\{1,\min_{1\le j\le J}
\beta_j \right\},\quad \beta_j =
\begin{cases}
\dfrac{M_j-u_0-\bm a_j^\top\bm L_k}{\bm a_j^\top\bm p_k},
& \bm a_j^\top\bm p_k > 0, \\
\dfrac{m_j-u_0-\bm a_j^\top\bm L_k}{\bm a_j^\top\bm p_k},
& \bm a_j^\top\bm p_k < 0, \\
+\infty, & \bm a_j^\top\bm p_k=0.
\end{cases}
\]
If $\alpha_k<1$, then there exists an index $j$ such that 
$\beta_j=\alpha_k$. In this case the active indicator $\delta_j$ is adjusted 
to $-\sgn(\bm a_j^\top\bm p_k)$ accordingly.

On the other hand, if $\bm p_k$ is zero, then we check the signs of the 
Lagrange multipliers. If all the multipliers are non-negative, then we have 
achieved optimality. Otherwise, we can find a feasible direction by dropping 
the constraint corresponding to the most negative multiplier.

The initial iterate $\bm L_0$ can be any feasible solution of \eqref{eq:dic} 
such as the null solution $\bm L_0=\bm 0$.  And the initial active 
indicators $\delta_j$ are all set to zero. The whole algorithm is sketched 
out in Appendix A.  We refer the readers to \cite{Nocedal2006} for more 
details about the active-set methods.

\begin{remark}[Connection with MC limiter]
  Let us investigate a special case: the ILR on a one-dimensional grids with 
  uniform spacing $h$. For each cell $T_0$, let the left and right neighbors 
  be labeled with 1 and 2 respectively. Furthermore, the backward and 
  forward slopes are denoted by $\sigma_1$ and $\sigma_2$ 
  respectively, i.e.
  \[
  \sigma_1=\dfrac{u_1-u_0}{x_1-x_0}=\dfrac{u_0-u_1}{h},\quad 
  \sigma_2=\dfrac{u_2-u_0}{x_2-x_0}=\dfrac{u_2-u_0}{h}.
  \]
  The problem \eqref{eq:dic} then reduces to the following univariate QP 
  problem
  \begin{equation}
  	\begin{aligned}
  		\min~ & L^2-(\sigma_1+\sigma_2)L \\
  		\st~ & 2\sigma_1^-\le L \le 2\sigma_1^+, \\
  		&2\sigma_2^-\le L \le 2\sigma_2^+ .
  	\end{aligned}
  	\label{eq:mclimiter}
  \end{equation}
  The optimal solution of \eqref{eq:mclimiter}
  \[
  L^*=\minmod
  \left(2\sigma_1,2\sigma_2,\dfrac{\sigma_1+\sigma_2}{2}\right),
  \]
  is the well-known \emph{MC (monotonized central-difference) limiter}, 
  where the \emph{minmod} function
  \[
  \minmod(a,b,c)
  =
  \begin{cases}
  \min\{|a|,|b|,|c|\}\sgn(a), & 
  \text{if~} a,b \text{~and~} c \text{~have the same
  	sign},\\
  0, & \text{otherwise}.
  \end{cases}
  \]	
\end{remark}

\begin{remark}[Connection with Barth's limiter]
  If the gradient is further restricted to the direction of unlimited gradient,
  i.e. $\bm L=-\phi \mathbf G^{-1}\bm c$ for some $\phi\in[0,1]$, then
  we arrive at the \emph{Barth's limiter} \cite{Barth1989}. A simple 
  derivation yields the following explicit expression for $\phi$
  \[
  \phi=\min_{1\le j\le J}
  \begin{cases}
  \dfrac{M_j-u_0}{-\bm c^\top\mathbf G^{-1}\bm a_j},
  & \text{if~}-\bm c^\top\mathbf G^{-1}\bm a_j > M_j-u_0, \\
  \dfrac{m_j-u_0}{-\bm c^\top\mathbf G^{-1}\bm a_j},
  & \text{if~}-\bm c^\top\mathbf G^{-1}\bm a_j < m_j-u_0, \\
  1, & \text{otherwise.} 
  \end{cases}
  \]
  A comparison of ILR against Barth's limiter will be provided in Section 5.	
\end{remark}

\subsection{Boundary treatment}\label{sec:bndtr}

For the cells located around the domain boundary, the proposed 
reconstruction can not be implemented smoothly according to the procedure 
introduced in the previous subsection.  On the one hand, for those cells with 
at least one edge on the boundary, the von Neumann neighbors are 
inadequate to give a reasonable reconstruction (Fig. \ref{fig:bctreat} (a)).  
Thus, besides the von Neumann neighbors, we include those cells sharing
at least one vertex with the current cell, which introduces the so-called 
\emph{Moore neighbors} $T_{M_1},T_{M_2},\cdots,T_{M_K}$
(Fig. \ref{fig:bctreat} (b)). On the other hand, it is hard to specify a suitable 
lower and upper bounds for edges on the boundary when the solution from 
the outside of the computational domain is unavailable. Here we adjust the 
bounds using solutions from the entire stencil, i.e.
\[
m_j=\min\{u_0,u_{M_1},\cdots,u_{M_K}\},\quad 
M_j=\max\{u_0,u_{M_1},\cdots,u_{M_K}\},\quad 
j=1,2,\cdots,J.
\]
The resulting QP problem is still \eqref{eq:dic}, with a new definition of 
$\mathbf G$ and $\bm c$ 
\begin{align*}
\mathbf G=\sum_{i=1}^K \bm r_{M_i}\bm r_{M_i}^\top,\quad 
\bm c= \sum_{i=1}^K (u_0-u_{M_i})\bm r_{M_i},\quad 
\bm r_{M_i}=\bm x_{M_i}-\bm x_0,\quad 
i=1,2,\cdots,K.
\end{align*}	
Now the flux computation can be computed in the usual way. The boundary 
condition is imposed on the midpoint of the boundary edge. When periodic 
boundary condition is prescribed, no special treatment is needed.
\begin{figure}[htbp]
	\centering
	\subfloat[Von Neumann neighbors]
	{\includegraphics[width=.45\textwidth]{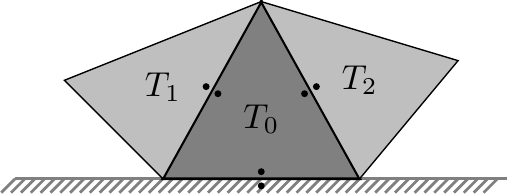}}
	\qquad
	\subfloat[Moore neighbors]
	{\includegraphics[width=.45\textwidth]{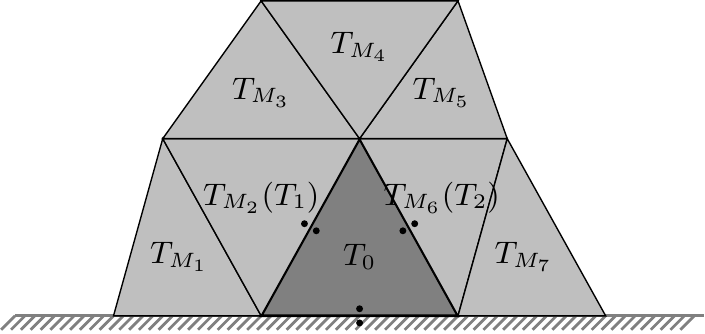}}
	\caption{Boundary treatment.}\label{fig:bctreat}
\end{figure}	

\section{Local maximum principle and positivity-preserving property}

\subsection{Local maximum principle}

A more appealing feature of our reconstruction is that the resulting finite 
volume scheme satisfies a \emph{local maximum principle}. We restrict 
ourselves to first-order Euler forward in time
\begin{equation}
u_0^{n+1}= u_0^n-\dfrac{\Delta t_n}{|T_0|}
\sum_{j=1}^J \mathcal F(u_{j}^-,u_{j}^+;\bm n_j)|e_j|.
\label{fvm:second}
\end{equation}
Second order SSP Runge-Kutta time discretization \eqref{eq:TVDRK2} will 
definitely keep the validity of the local maximum principle since it is a 
convex combination of two Euler forward steps. For clarity, the labels of 
edges are ordered in the way in Fig. \ref{fig:quadconfig}. Here we have the 
following theorem on local maximum principle. The technique to prove the 
local maximum principle is similar to that of \cite{Zhang2012}.
\begin{figure}[t]
	\centering
 	\subfloat[Triangle]
 	{\includegraphics[width=.3\textwidth]{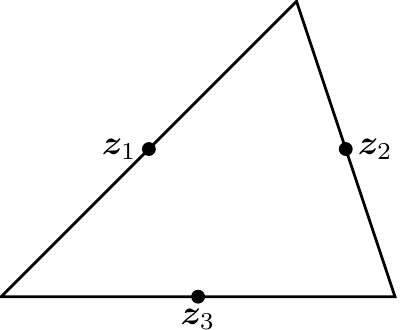}}
 	\qquad
 	\subfloat[Twin triangle]
 	{\includegraphics[width=.3\textwidth]{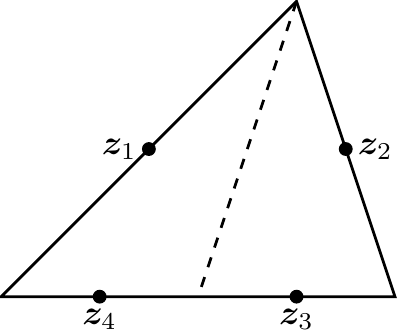}}
 	\caption{Labels of quadrature points for triangles and twin triangles.}
 	\label{fig:quadconfig}
\end{figure}
\begin{theorem}[Local maximum principle]\label{thm:lmp}
Suppose that $\mathcal F$ is a monotone Lipschitz continuous numerical 
flux function. Let $\mathcal T$ be a triangulation mixed with triangles and 
twin triangles. If for each cell $T_0\in\mathcal T$ the linear reconstruction 
satisfies
\[
\min\{u_0^n,u_j^n\}\le u_j^-\le \max\{u_0^n,u_j^n\},\quad 
j=1,2,\cdots,J,
\] 
then the finite volume scheme \eqref{fvm:second} fulfills the local maximum 
principle
\[
\min\{u_0^n,u_1^n,\cdots,u_J^n\}\le u_0^{n+1}
\le \max\{u_0^n,u_1^n,\cdots,u_J^n\},
\]
under the CFL-like condition
\[
\Delta t_n
\sup_{u^-,u^+,\bm n}
\left|\dfrac{\partial\mathcal F(u^-,u^+;\bm n)}{\partial u^+}\right|
\le \dfrac{1}{12}h,
\]
where $h$ is the minimum size of cells measured by their diameters of 
inscribed circles.
\end{theorem}
\begin{proof}
From Fig. \ref{fig:quadconfig} we have the following formula
\begin{equation}
\dfrac{1}{|T_0|}\int_{T_0}\hat u_0(\bm x)\mathrm d\bm x
=\begin{cases}
\dfrac{1}{3} \hat u_0(\bm z_1)+\dfrac{1}{3} \hat u_0(\bm z_2)
+\dfrac{1}{3} \hat u_0(\bm z_3),
& T_0 \text{~is a triangle},\\
\dfrac{1}{3}\hat u_0(\bm z_1)+\dfrac{1}{3}\hat u_0(\bm z_2)+
\dfrac{1}{6}\hat u_0(\bm z_3)+\dfrac{1}{6}\hat u_0(\bm z_4),
& T_0 \text{~is a twin triangle}.
\end{cases}
\label{eq:linquad}
\end{equation}
As a result, the finite volume scheme \eqref{fvm:second} can be split as
\begin{equation}
\begin{split}
u_0^{n+1}&= u_0^n-\dfrac{\Delta t_n}{|T_0|}
\sum_{j=1}^3 \mathcal F(u_{j}^-,u_{j}^+;\bm n_j)|e_j| \\
&=\dfrac{1}{3}(u_1^-+u_2^-+u_3^-)
-\dfrac{\Delta t_n}{|T_0|}
\sum_{j=1}^3 \mathcal F(u_{j}^-,u_{j}^+;\bm n_j)|e_j| \\
&=\dfrac{1}{3} (H_{1}+H_{2}+H_{3}),
\end{split}
\label{eq:scheme1}
\end{equation}
with
\begin{equation*}
\begin{split}
H_{1}&=u_{1}^--\dfrac{3\Delta t_n}{|T_0|}
\left[\mathcal F(u_{1}^-,u_{1}^+;\bm n_1)|e_1|+
\mathcal F(u_{1}^-,u_{2}^-;\bm n_2)|e_2|+
\mathcal F(u_{1}^-,u_{3}^-;\bm n_3)|e_3|\right],\\
H_{2}&=u_{2}^--\dfrac{3\Delta t_n}{|T_0|}
\left[\mathcal F(u_{2}^-,u_{2}^+;\bm n_2)+
\mathcal F(u_{2}^-,u_{1}^-;-\bm n_2)\right]|e_2|,\\
H_{3}&=u_{3}^--\dfrac{3\Delta t_n}{|T_0|}
\left[\mathcal F(u_{3}^-,u_{3}^+;\bm n_3)+
\mathcal F(u_{3}^-,u_{1}^-;-\bm n_3)\right]|e_3|,
\end{split}
\end{equation*} 
if $T_0$ is a triangle, or
\begin{equation}
  \begin{split}
    u_0^{n+1}&= u_0^n-\dfrac{\Delta t_n}{|T_0|}
    \sum_{j=1}^4 \mathcal F(u_{j}^-,u_{j}^+;\bm n_j)|e_j| \\
    &=\dfrac{1}{6}(2u_1^-+2u_2^-+u_3^-+u_4^-)
    -\dfrac{\Delta t_n}{|T_0|}
    \sum_{j=1}^4 \mathcal F(u_{j}^-,u_{j}^+;\bm n_j)|e_j| \\
    &=\dfrac{1}{6} (2H_{1}+2H_{2}+H_{3}+H_{4}),
  \end{split}
\label{eq:scheme2}
\end{equation}
with
\begin{equation*}
   \begin{split}
    H_{1}&=u_{1}^--\dfrac{3\Delta t_n}{|T_0|}
    \big[\mathcal F(u_{1}^-,u_{1}^+;\bm n_1)|e_1|+
    \mathcal F(u_{1}^-,u_{2}^-;\bm n_2)|e_2|\\
    &+\mathcal F(u_{1}^-,u_{3}^-;\bm n_3)|e_3|+
    \mathcal F(u_{1}^-,u_{4}^-;\bm n_3)|e_4|\big],\\
    H_{2}&=u_{2}^--\dfrac{3\Delta t_n}{|T_0|}
    \left[\mathcal F(u_{2}^-,u_{2}^+;\bm n_2)+
    \mathcal F(u_{2}^-,u_{1}^-;-\bm n_2)\right]|e_2|,\\
    H_{3}&=u_{3}^--\dfrac{6\Delta t_n}{|T_0|}
    \left[\mathcal F(u_{3}^-,u_{3}^+;\bm n_3)+
    \mathcal F(u_{3}^-,u_{1}^-;-\bm n_3)\right]|e_3|, \\
    H_{4}&=u_{4}^--\dfrac{6\Delta t_n}{|T_0|}
    \left[\mathcal F(u_{4}^-,u_{4}^+;\bm n_3)+
    \mathcal F(u_{4}^-,u_{1}^-;-\bm n_3)\right]|e_4|,
  \end{split}
\end{equation*} 
if $T_0$ is a twin triangle.

It follows that the derivatives of $H_1,\cdots,H_J$ with respect to their 
arguments $u_1^\pm,\cdots,u_J^\pm$ are all non-negative provided that
\[
\dfrac{\Delta t_n}{|T_0|}
\sup_{u^-,u^+,\bm n}
\left|\dfrac{\partial\mathcal F(u^-,u^+;\bm n)}{\partial u^+}\right|
\cdot
\sum_{j=1}^J |e_j|
\le \dfrac{1}{3},
\]
or, using the definition of $h$,
\[
\Delta t_n
\sup_{u^-,u^+,\bm n}
\left|\dfrac{\partial\mathcal F(u^-,u^+;\bm n)}{\partial u^+}\right|
\le \dfrac{1}{12}h.
\]
Therefore, if we view the right-hand side of the scheme \eqref{eq:scheme1}
or \eqref{eq:scheme2} as a function of $u_{1}^\pm,\cdots,u_{J}^\pm$ 
\[
u_0^{n+1}=H(u_{1}^-,\cdots,u_{J}^-,
u_{1}^+,\cdots,u_{J}^+),
\]
then the function $H$ is non-decreasing with respect to each argument.
Moreover, we have $H(u,\cdots,u)=u$ due to the consistency of numerical 
flux function. 

Denote 
\[
u^{\min}=\min\{u_0^n,u_1^n,\cdots,u_J^n\},\quad 
u^{\max}=\max\{u_0^n,u_1^n,\cdots,u_J^n\},
\]
then the trace values satisfy the inequalities
\[
 u^{\min}\le \min\{u_0,u_j\}\le u_{j}^\pm 
 \le \max\{u_0,u_j\}\le u^{\max},\quad 
 j=1,2,\cdots,J.
\]
Consequently the monotonicity of $H$ implies the local maximum principle 
\[
u^{\min}=H(u^{\min},\cdots,u^{\min})\le u_0^{n+1}\le 
H(u^{\max},\cdots,u^{\max})=u^{\max}.
\]
\end{proof}

\begin{remark}
  In Theorem \ref{thm:lmp}, we focus on the case with triangular meshes 
  generated by the adaptive mesh refinement.  Indeed the local maximum 
  principle is valid on any polygon meshes, as long as a quadrature formula 
  like \eqref{eq:linquad} can be found.
\end{remark}

\subsection{Positivity-preserving property}

In this part we extend the previous framework of scalar conservation laws to 
the system of the Euler equations of ideal gases:
\[
\dfrac{\partial \bm u}{\partial t}
+\nabla\cdot\bm F(\bm u)=\bm 0,
\]
\[
\bm u=
\begin{pmatrix}
\rho \\ \rho u \\ \rho v \\ E
\end{pmatrix},\quad
\bm F(\bm u)=
\begin{pmatrix}
\rho u & \rho v \\
\rho u^2+p & \rho uv \\
\rho uv    & \rho v^2+p \\
u(E+p)    & v(E+p)
\end{pmatrix},\quad 
p=(\gamma-1)
\left(E-\dfrac{1}{2}\rho|\bm v|^2\right),
\]
where $\rho$ is the density, $\bm v=(u,v)$ is the velocity, $E$ is the total 
energy, $p$ is the pressure and $\gamma>1$ is a constant.

The  local maximum principle is no longer applicable for system of
conservation laws. However, the positivity-preserving property is desired in 
the computation of the Euler equations. It means that the solution $\bm u$ 
belongs to the \emph{set of admissible states}
\[
\mathcal G=
\left\{
[\rho,~m,~n,~E]^\top:
\rho > 0\text{~and~}p=(\gamma-1)\left(E-\dfrac{m^2+n^2}{2\rho}\right)>0
\right\}.
\]

To simplify the discussion we consider the finite volume scheme
\begin{equation}
\bm u_0^{n+1}= \bm u_0^n-\dfrac{\Delta t_n}{|T_0|}
\sum_{j=1}^J \mathcal F(\bm u_{j}^-,\bm u_{j}^+;\bm n_j)|e_j|,
\label{fvm:Euler}
\end{equation}
with \emph{local Lax-Friedrichs flux}
\begin{equation}
\mathcal F(\bm u_j^-,\bm u_j^+;\bm n_j) = \dfrac12 
[\bm F(\bm u^-_j) + \bm F(\bm u^+_j)] \cdot \bm n_j - 
\dfrac12 a(\bm u^+_j-\bm u^-_j),\quad
a=\||\bm v|+c\|_\infty,
\label{eq:llf}
\end{equation}
where $c=\sqrt{{\gamma p}/{\rho}}$ is the speed of sound.

To guarantee positivity, we perform the reconstruction with respect to the 
\emph{primitive} variables $\rho,u,v$ and $p$ respectively. The following 
theorem states that the finite volume scheme will preserve the positivity of 
numerical solutions.
\begin{theorem}[Positivity-preserving]\label{thm:pp}
  Let $\mathcal T$ be a triangulation mixed with triangles and twin 
  triangles. Labels of edges are illustrated in Fig. \ref{fig:quadconfig}. 
  Suppose that for each cell $T_0\in\mathcal T$ the linear reconstruction 
  with respect to the primitive variables satisfies 
  $\bm u_j^-\in\mathcal G$ for $j=1,2,\cdots,J$. Then the finite volume 
  scheme \eqref{fvm:Euler} with local Lax-Friedrichs flux \eqref{eq:llf} 
  satisfies $\bm u_0^{n+1}\in\mathcal G$ under the CFL condition
  \begin{equation}
  	a\Delta t_n\le \dfrac{1}{2}\beta h,
  	\label{eq:cflbeta}
  \end{equation}
  where $h$ has the same meaning as Theorem \eqref{thm:lmp} and
  $0< \beta \le 1$ is sufficiently small such that 
  \[
  \bm u_0^n-\beta (\bm u_1^-+\bm u_2^-+\bm u_3^-)\in \mathcal G,
  \]
  for any triangle $T_0$ and
  \[
  \bm u_0^n-\beta \left(\bm u_1^-+\bm u_2^-+\dfrac{1}{2}
  \bm u_3^-+\dfrac{1}{2}\bm u_4^-\right)\in \mathcal G,
  \] 
  for any twin triangle $T_0$.
\end{theorem}
\begin{proof}
  We only prove the case with triangle here. The case with twin triangle can 
  be carried out in a similar manner. Substituting \eqref{eq:llf} into 
  \eqref{fvm:Euler} and rearranging terms yield
  \begin{equation}
  	\begin{split}
  		\bm u_0^{n+1}=&~\bm u_0^n -
  		\dfrac{\Delta t_n}{|T_0|}
  		\left[\mathcal F(\bm u_{1}^-,\bm u_{1}^+;\bm n_1)|e_1|+
  		\mathcal F(\bm u_{2}^-,\bm u_{2}^+;\bm n_2)|e_2|+
  		\mathcal F(\bm u_{2}^-,\bm u_{3}^+;\bm n_3)|e_3|\right] \\
  		=&~ \bm u_0^n-\lambda(|e_1|+|e_2|+|e_3|)
  		(\bm u_1^-+\bm u_2^-+\bm u_3^-)
  		+\lambda(|e_1|+|e_3|-|e_2|)\bm u_2^- \\&
  		+ \lambda(|e_1|+|e_2|-|e_3|)\bm u_3^-
  		+\lambda|e_1|\left(\bm u_1^+-\dfrac{1}{a}
  		\bm F(\bm u_1^+)\cdot\bm n_1\right)\\&
  		+\lambda|e_2|\left(\bm u_2^--\dfrac{1}{a}
  		\bm F(\bm u_2^-)\cdot\bm n_2\right)
  		+\lambda|e_3|\left(\bm u_3^--\dfrac{1}{a}
  		\bm F(\bm u_3^-)\cdot\bm n_3\right)\\
  		&+\lambda |e_2|\left(\bm u_2^+-\dfrac{1}{a}
  		\bm F(\bm u_2^+)\cdot \bm n_2\right)
  		+ \lambda |e_2|\left(\bm u_1^-+\dfrac{1}{a}
  		\bm F(\bm u_1^-)\cdot \bm n_2\right)\\
  		&+\lambda |e_3|\left(\bm u_3^+-\dfrac{1}{a}
  		\bm F(\bm u_3^+)\cdot \bm n_3\right)
  		+ \lambda |e_3|\left(\bm u_1^-+\dfrac{1}{a}
  		\bm F(\bm u_1^-)\cdot \bm n_3\right),
  	\end{split}
  	\label{eq:ppsplit}
  \end{equation}
  where $\lambda = a\Delta t_n/2|T_0|$.
  
  It can be verified that the set $\mathcal G$ is a \emph{convex cone}. 
  Note that $\lambda(|e_1|+|e_2|+|e_3|)={2a\Delta t_n}/{h}\le \beta$.
  Since $\bm u_1^{-},\bm u_2^-,\bm u_3^-\in\mathcal G$ and 
  $\bm u_0^n-\lambda(|e_1|+|e_2|+|e_3|)
  (\bm u_1^-+\bm u_2^-+\bm u_3^-)\in\mathcal G$,
  to guarantee $\bm u_0^{n+1}\in\mathcal G$ it suffices to verify that 
  \[
  \bm u \pm \dfrac{1}{a} \bm F(\bm u)\cdot\bm n\in\mathcal G,
  \]
  for any $\bm u\in\mathcal G$ and normal vector $\bm n$, that is, to 
  check the positivity of the \emph{density} and \emph{pressure}. Actually,
  \begin{align*}
  	\rho\left(\bm u\pm\dfrac{1}{a}\bm F(\bm u)\cdot\bm n\right)
  	&=
  	\dfrac{\rho}{a}(a\pm\bm v\cdot\bm n)>0,\\
  	p\left(\bm u\pm\dfrac{1}{a}\bm F(\bm u)\cdot\bm n\right)
  	&=
  	\dfrac{
  		(\gamma+1)(a\pm \bm v\cdot\bm n)^2
  		+(\gamma-1)[(a\pm \bm v\cdot\bm n)^2-c^2]
  	}
  	{2\gamma a(a\pm \bm v\cdot\bm n)}p>0,
  \end{align*}
  where $\rho(\bm u)=\rho$ and $p(\bm u)=(\gamma-1)
  \left(E-\dfrac{m^2+n^2}{2\rho}\right)$ for 
  $\bm u=[\rho,~m,~n,~E]^\top$.
\end{proof}

\begin{remark}
  Since $\bm u_1^-,\bm u_2^-,\cdots,\bm u_J^-$ are close to the cell 
  average $\bm u_0^n$, we can expect that the value of the constant 
  $\beta$ in \eqref{eq:cflbeta} is close to $1/3$.
\end{remark}

\begin{remark}
  Assume that at $n$-th time level $\bm u_h^n\in\mathcal G$. Then for 
  any cell $T_0\in\mathcal T$, the primitive variables $\rho_0^n$ and 
  $p_0^n$ must be positive. The constraints of ILR guarantee the positivity 
  of trace values $\rho_j^{-}$ and $p_j^{-}$ for all $j$. As a result, 
  $\bm u_j^-\in\mathcal G$ for all $j$. From Theorem \ref{thm:pp} we 
  conclude that $\bm u_0^{n+1}\in\mathcal G$ and hence at $(n+1)$-th 
  time level $\bm u_h^{n+1}\in\mathcal G$. Therefore, the ILR for the 
  Euler equations leads to positive numerical solutions as long as the 
  initial solution is positive.
\end{remark}

\section{Numerical results}

In this section we present several numerical examples to demonstrate the 
numerical performances of our integrated linear reconstruction on various 
unstructured grids. 

\subsection{Linear advection problem}

We consider the linear advection equation
\[
u_t+u_x+2u_y=0,
\]
with initial profile given by the double sine wave function 
\cite{Hubbard1999, Park2010, May2013, Chen2016}
\[
u_0(x,y)=\sin(2\pi x)\sin(2\pi y).
\]

The computational domain is $[0,1]\times[0,1]$. The periodic boundary 
condition is applied. We perform the convergence test on both uniform and 
non-uniform meshes (see Fig. \ref{fig:tri} for the coarse mesh). The uniform 
mesh is generated by dividing rectangular cells along the diagonal direction,
while the non-uniform mesh is generated by the Delaunay triangulation of
EasyMesh \cite{EasyMesh}. The upwind flux is adopted as the numerical flux and
the CFL number is $0.3$. We present quantitative comparisons of solution errors
as well as the reconstruction cost until $t=1$ in Tables \ref{tab:sinewavere}
and \ref{tab:sinewaveir}, where we include the MLP limiter (version u1)
\cite{Park2010} as a representative example of limiters involving the Moore
neighbors. It is found that even though MLP is slightly better than our ILR,
they behave quite similar in the accuracy and efficiency. As a comparison,
Barth's limiter has the advantage on the efficiency, but it only gives rise to
first-order accuracy. Moreover, it was observed in this numerical test that the
iteration of QP solver converges in roughly two steps on average, which
confirms the excellent performance of ILR.

\begin{figure}[htbp]
	\centering
	\subfloat[Uniform mesh ($8\times 8\times 2$ cells)]
	{\includegraphics[width=.35\textwidth]{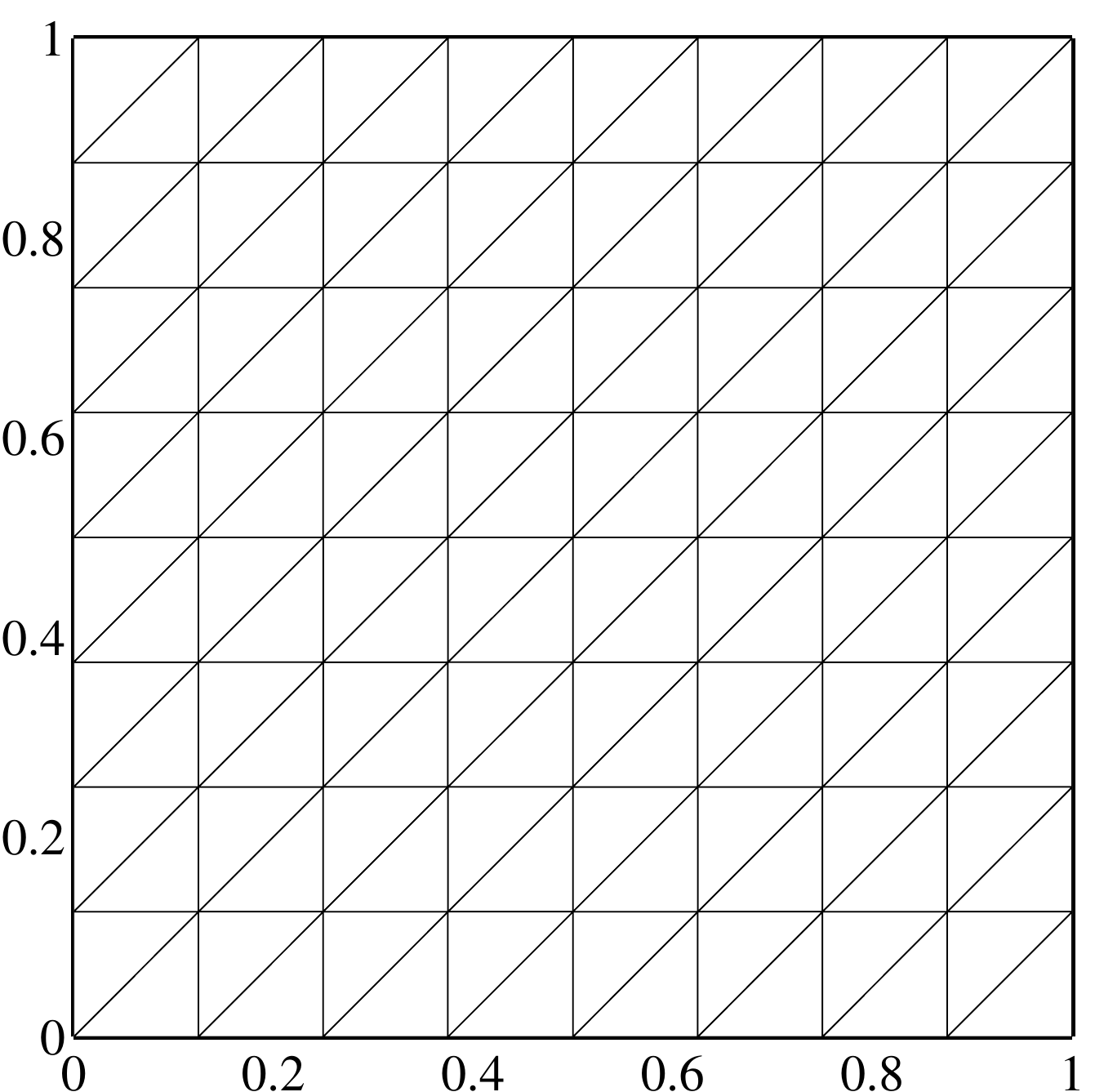}}
	\qquad
	\subfloat[Non-uniform mesh (64 cells)]
	{\includegraphics[width=.35\textwidth]{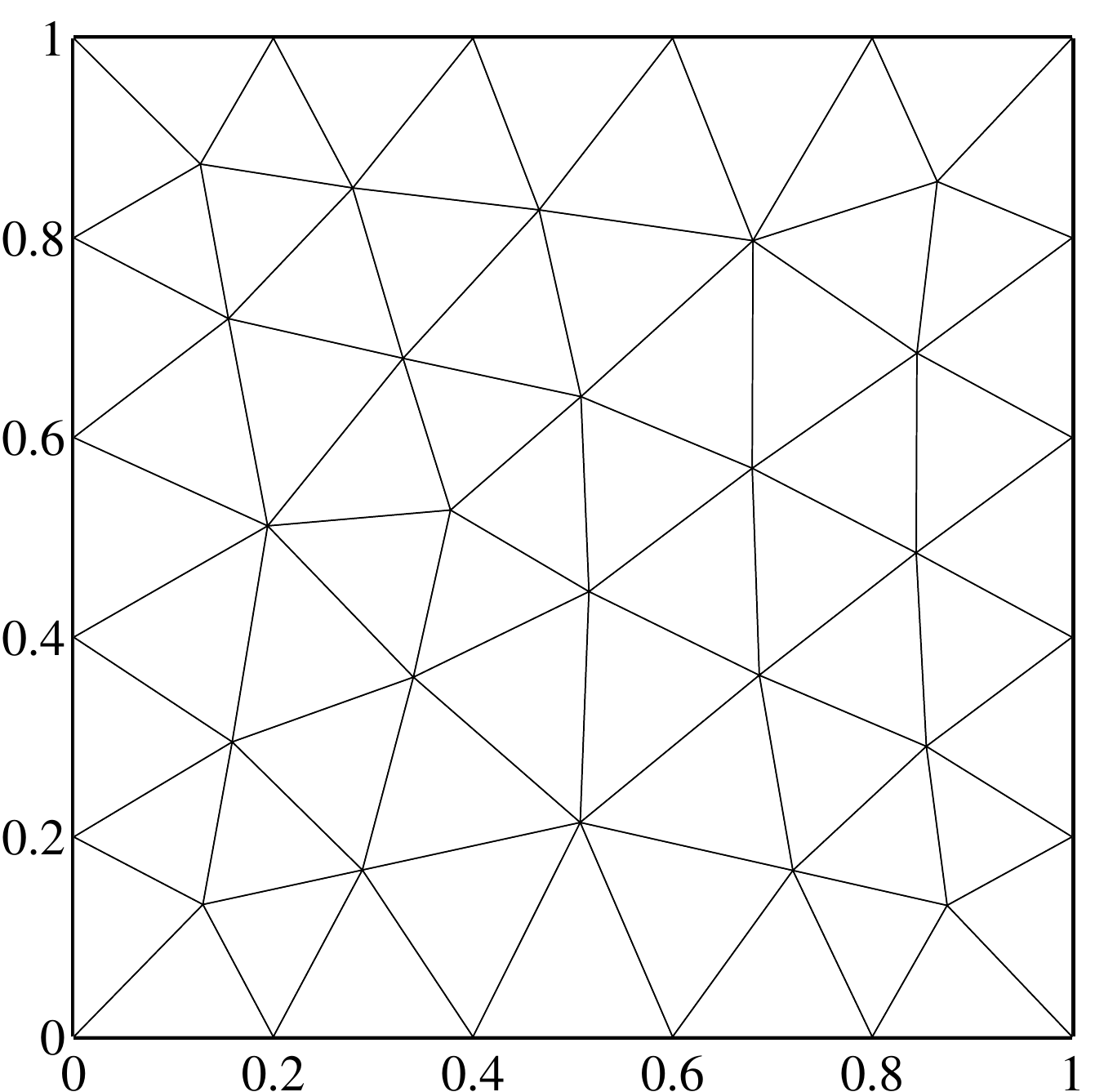}}
	\caption{Coarse meshes used in the convergence test.}
	\label{fig:tri}
\end{figure}

\begin{table}[htbp]
   \centering
   \caption{Convergence test for the advection of double sine wave: uniform 
   	            meshes.}\label{tab:sinewavere}
	\begin{tabular}{llllllr}
     \toprule 
     & Cells & $L^1$ error & Order &
     $L^\infty$ error & Order & Reconstruction cost (s) \\ \midrule
     ILR &
     16$\times$16$\times$2 &  4.06e-2 & 2.33  & 1.71e-1 & 1.45  & 1 \\ &
     32$\times$32$\times$2 & 1.16e-2 & 1.81  & 6.35e-2 & 1.43  & 7 \\ &
     64$\times$64$\times$2 & 3.28e-3 & 1.82  & 3.00e-2 & 1.08  & 64 \\ &
     128$\times$128$\times$2 & 8.61e-4 & 1.93  & 1.23e-2 & 1.28  & 444 \\ \midrule
     MLP & 16$\times$16$\times$2 & 6.19e-2 & 1.96  & 2.54e-1 & 1.13  & 1 \\ &
     32$\times$32$\times$2 & 1.32e-2 & 2.23  & 9.70e-2 & 1.39  & 6 \\ &
     64$\times$64$\times$2 & 2.69e-3 & 2.29  & 3.45e-2 & 1.49  & 51 \\ &
     128$\times$128$\times$2 & 5.71e-4 & 2.24  & 1.19e-2 & 1.54  & 392 \\ \midrule
     Barth's limiter & 16$\times$16$\times$2& 1.36e-1 & 1.22  & 4.25e-1 & 0.82  & 1 \\ &
     32$\times$32$\times$2 & 7.26e-2 & 0.90  & 2.65e-1 & 0.68  & 3 \\ &
     64$\times$64$\times$2 & 3.46e-2 & 1.07  & 1.40e-1 & 0.92  & 29 \\ &
     128$\times$128$\times$2 & 1.82e-2 & 0.93  & 9.93e-2 & 0.50  & 231 \\ \midrule
     Unlimited & 16$\times$16$\times$2 & 3.40e-2 & 2.17  & 7.51e-2 & 2.01  & 1 \\ &
     32$\times$32$\times$2 & 7.63e-3 & 2.16  & 1.69e-2 & 2.15  & 3\\ &
     64$\times$64$\times$2 & 1.85e-3 & 2.05  & 3.94e-3 & 2.10  & 27 \\ &
     128$\times$128$\times$2 & 4.58e-4 & 2.01  & 9.48e-4 & 2.05  & 203 \\ \bottomrule
	\end{tabular}
\medskip
\caption{Convergence test for the advection of double sine wave:   
	         non-uniform meshes.}\label{tab:sinewaveir}
	\begin{tabular}{llllllr}
	\toprule 
	& Cells & $L^1$ error & Order &
	$L^\infty$ error & Order & Reconstruction cost (s) \\ \midrule 
	ILR & 1024\hspace{3em} & 1.65e-2 & 2.11 & 1.01e-1 & 1.58 & 3 \\ 
	& 4096 & 4.30e-3 & 1.94 & 4.47e-2 & 1.17 & 24 \\
	& 16384 & 1.12e-3 & 1.95 & 1.73e-2 & 1.37 & 178 \\ 
	& 66536 & 2.97e-4 & 1.91 & 6.89e-3 & 1.33 & 1366 \\ \midrule 
	MLP & 1024 & 2.07e-2 & 2.15 & 1.41e-1 & 1.42 & 3 \\ 
	& 4096 & 4.38e-3 & 2.24 & 5.12e-2 & 1.46 & 22 \\ 
	& 16384 & 9.13e-4 & 2.26 & 1.77e-2 & 1.53 & 167 \\ 
	& 66536 & 2.01e-4 & 2.18 & 6.05e-3 & 1.55 & 1311 \\ \midrule 
	Barth's limiter & 1024 & 9.85e-2 & 1.11 & 3.34e-1 & 0.82 & 1 \\ 
	& 4096 & 4.59e-2 & 1.10 & 1.70e-1 & 0.97 & 12 \\ 
	& 16384 & 2.24e-2 & 1.04 & 9.33e-2 & 0.87 & 93 \\ 
	&66536 & 1.22e-2 & 0.88 & 6.09e-2 & 0.61 & 732 \\ \midrule 
	Unlimited & 1024 & 1.20e-2 & 2.25 & 3.53e-2 & 2.04 & 1 \\ 
	& 4096 & 2.83e-3 & 2.09 & 7.82e-3 & 2.18 & 11 \\ 
	& 16384 & 6.93e-4 & 2.03 & 1.98e-3 & 1.98 & 85 \\ 
	& 66536 & 1.72e-4 & 2.01 & 4.92e-4 & 2.01 & 677 \\ \bottomrule 
\end{tabular}
\end{table}

\subsection{Solid body rotation on the stretched mesh}

To assess the behavior of ILR for non-uniform scalar flow field against an 
anisotropic mesh, we consider the following solid body rotation problem
proposed by LeVeque \cite{LeVeque1996} 
\[
u_t-(y-0.5)u_x+(x-0.5)u_y=0,
\]
on a highly-stretched mesh in the computational domain $[0,1]\times[0,1]$
(Fig. \ref{fig:sbrmesh}). The highly-stretched mesh can be generated in the 
following three steps \cite{Diskin2012}:
\begin{enumerate}
	\item A rectangular mesh is stretched toward the horizontal line $y=0.5$ 
	         by a factor $\beta=1.2$;
	\item Irregularities are introduced by random shifts of interior nodes in   
	         $x$-directions;
	\item Each distorted quadrilateral is further divided along its diagonal    
	         direction.
\end{enumerate}

The initial profile consists of smooth hump, cone and slotted cylinder. 
These shapes are located within the circle of radius $r_0=0.15$ centered in 
$(x_0,y_0)$. In the other region the initial value vanishes. Denote by 
$r(x,y)=\sqrt{(x-x_0)^2+(y-y_0)^2}/r_0$ the normalized distance. The 
slotted cylinder is centered at $(x_0,y_0)=(0.5,0.75)$ and
\[
u_0(x,y)=
\begin{cases}
1, & \text{if~}|x-x_0|\ge 0.025 \text{~or~} y\ge 0.85,\\
0,  & \text{otherwise,}
\end{cases}
\]
the sharp cone is centered at $(x_0,y_0)=(0.5,0.25)$ and
\[
u_0(x,y)=1-r(x,y),
\]
and the smooth hump is centered at $(x_0,y_0)=(0.25,0.5)$ and
\[
u_0(x,y)=\dfrac{1+\cos(\pi r(x,y))}{4}.
\]

In the computation we use a fixed time step $\Delta t=0.6\pi h$. The 
homogeneous Dirichlet boundary conditions are prescribed. To measure the 
oscillation of the numerical solution we compute the discrete total variation 
(TV) of the piecewise \emph{linear} solution $\hat u_h=\displaystyle
\sum_{T_0\in\mathcal T}\hat u_0\mathbb 1_{T_0}$ in the following way
\[
\mathrm{TV}(\hat u_h) = \sum_{T_0\in\mathcal T} 
\left(\|\nabla \hat u_0\|\cdot|T_0|+\dfrac{1}{2}\sum_{j=1}^{J} 
|u_j^--u_j^+|\cdot|e_j|\right).
\]
The evolutions of the discrete total variation are depicted in Fig. 
\ref{fig:sbrresolution} for different resolutions and in Fig. \ref{fig:sbrlimiter} 
for different methods. From here we can see that the total variation is 
decreasing as a whole except for individual tiny fluctuations when the slot 
rotates towards the stretching line $y=0.5$. The total variation will flatten 
out as mesh refines. Moreover, our ILR has the best performance on 
preserving the total variation of the solution on the highly-stretched mesh 
among the three methods. The snapshots in Fig. \ref{fig:sbrcont} show the 
shapes of the solutions with different methods at $t = 2\pi$, corresponding 
to a complete rotation. We can observe that the result of ILR captures the  
initial profile much better than that of MLP or Barth's limiter.

\begin{figure}[htbp]
  \centering
  \begin{minipage}[b]{.24\linewidth}
   \includegraphics[width=\textwidth]{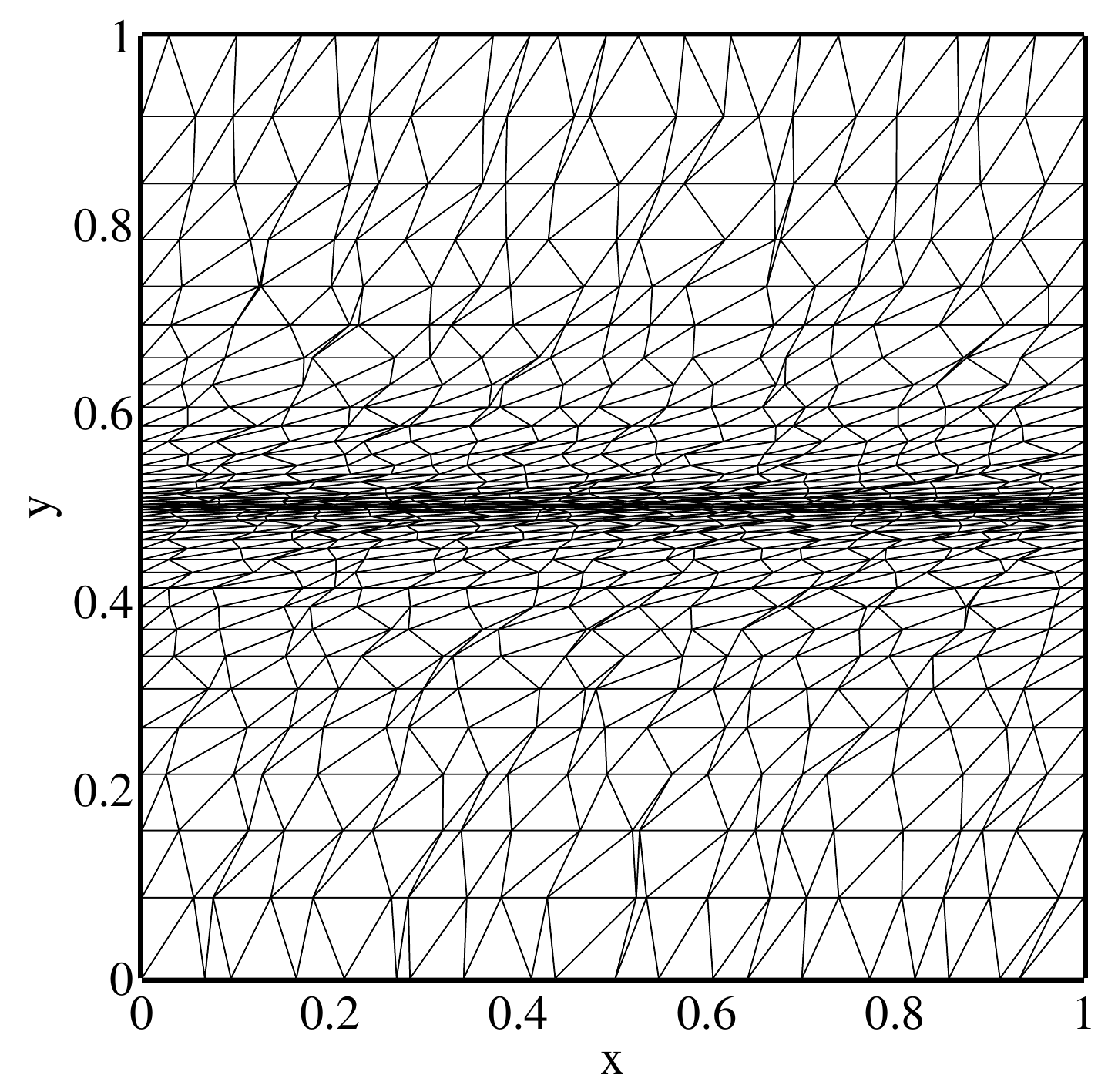}
   \caption{Stretched mesh with $20\times 40\times 2$ cells.}
   \label{fig:sbrmesh}
  \end{minipage}
  \quad
   \begin{minipage}[b]{.3\linewidth}
  \includegraphics[width=\textwidth]{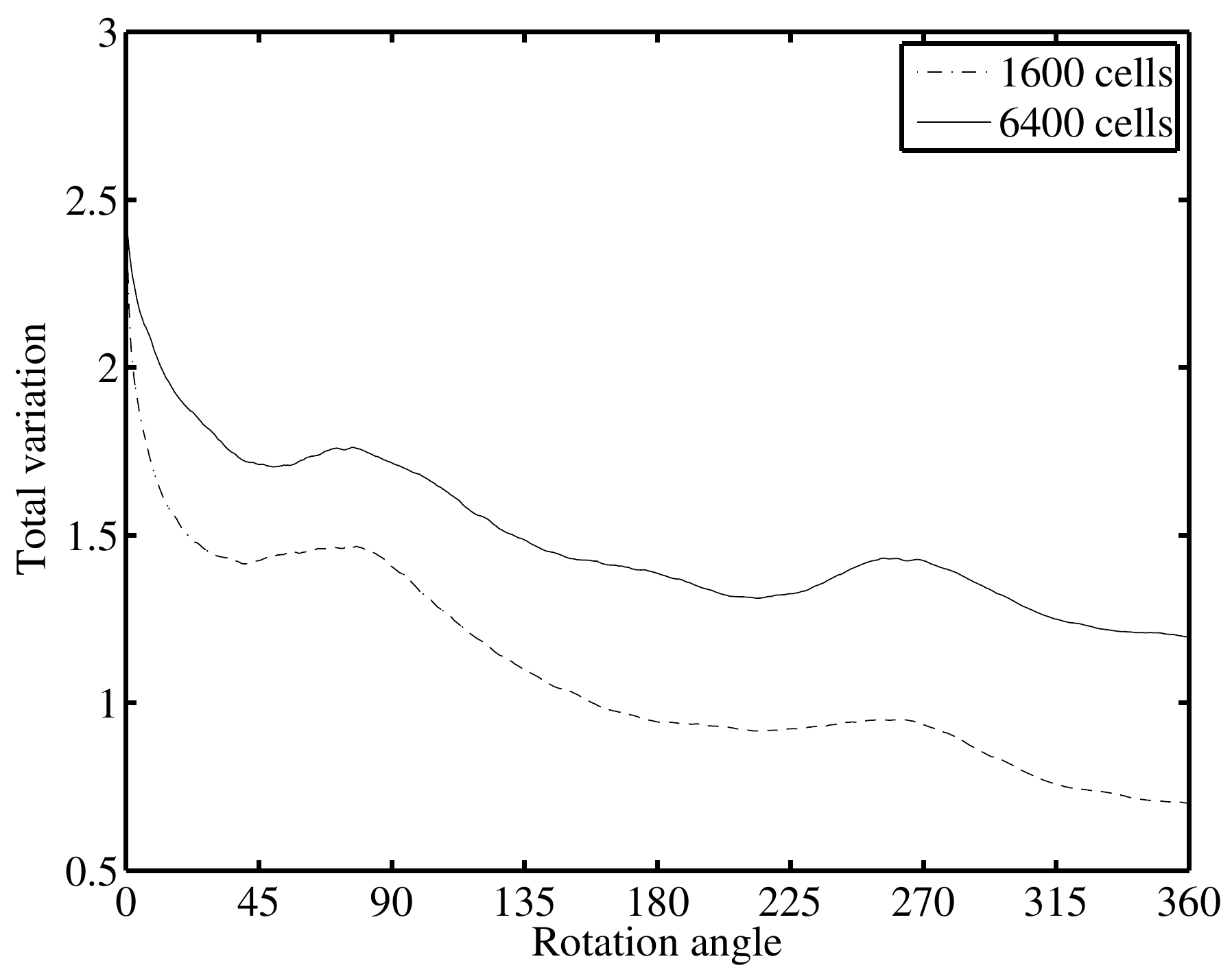}
  \caption{Total variation versus angle of rotation for different resolutions.}
  \label{fig:sbrresolution}
  \end{minipage}
   \quad
  \begin{minipage}[b]{.3\linewidth}
  \includegraphics[width=\textwidth]{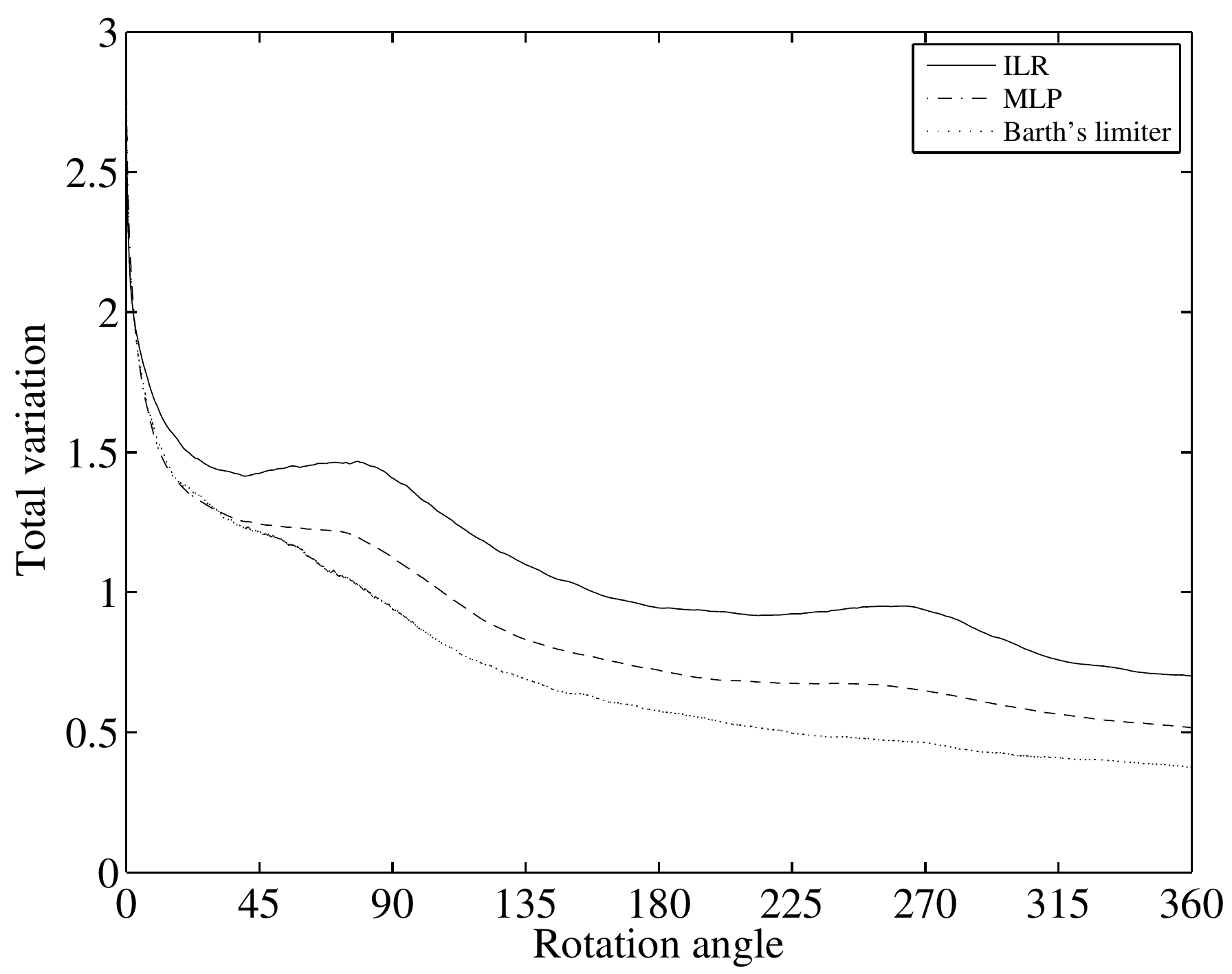}
  \caption{Total variation versus angle of rotation for different methods.}
  \label{fig:sbrlimiter}
\end{minipage}
  \subfloat{\includegraphics[width=.3\textwidth]{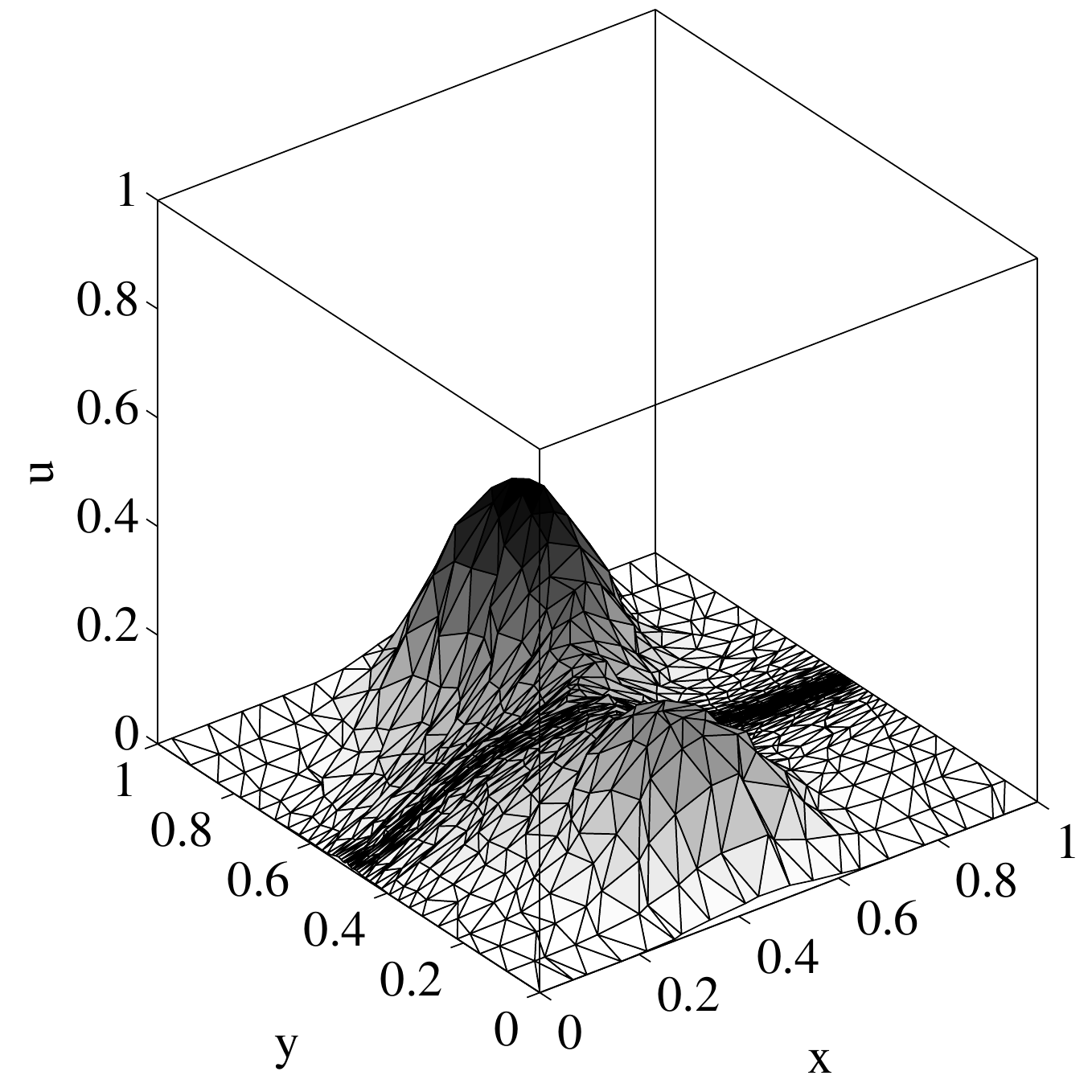}}
  \subfloat{\includegraphics[width=.3\textwidth]{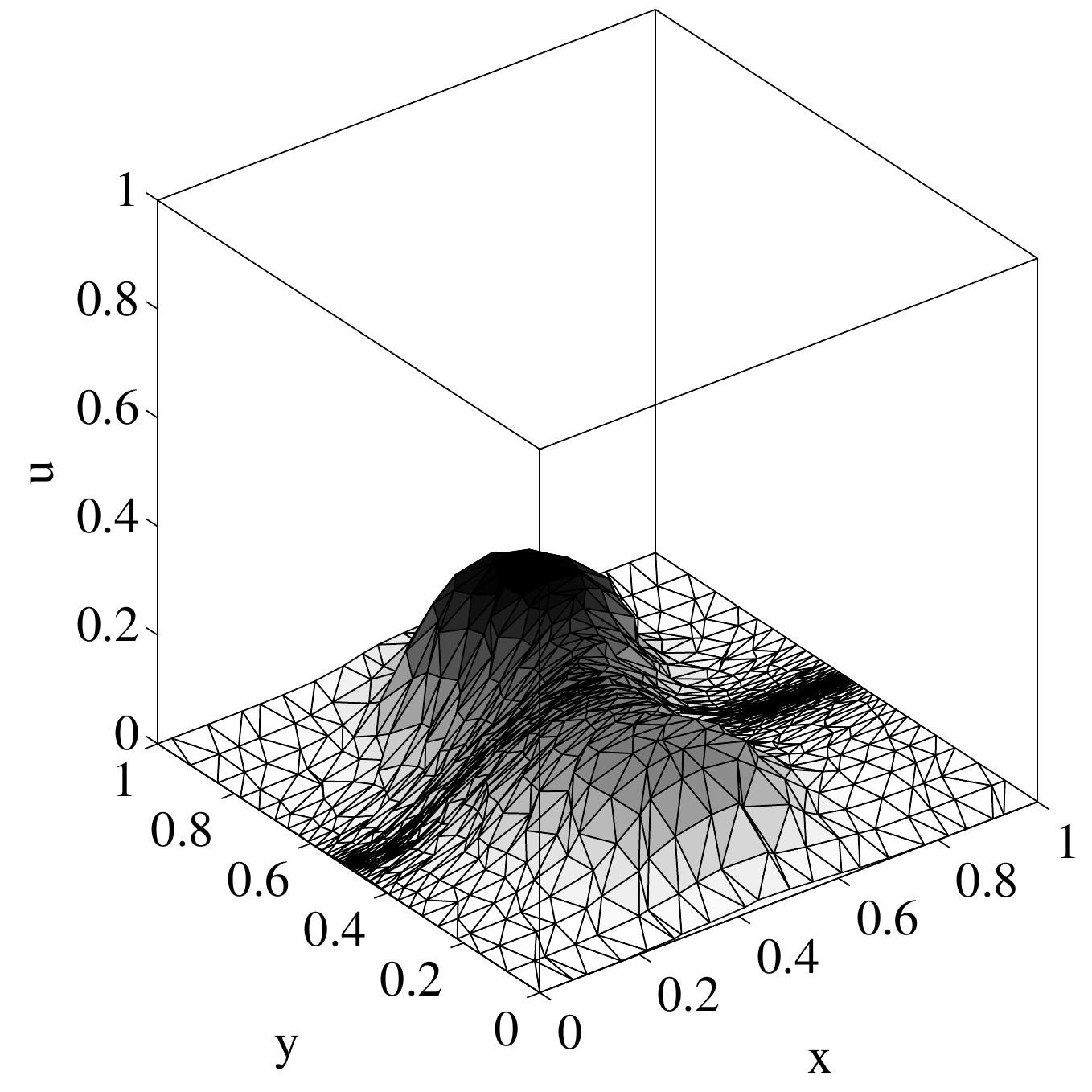}}
  \subfloat{\includegraphics[width=.3\textwidth]{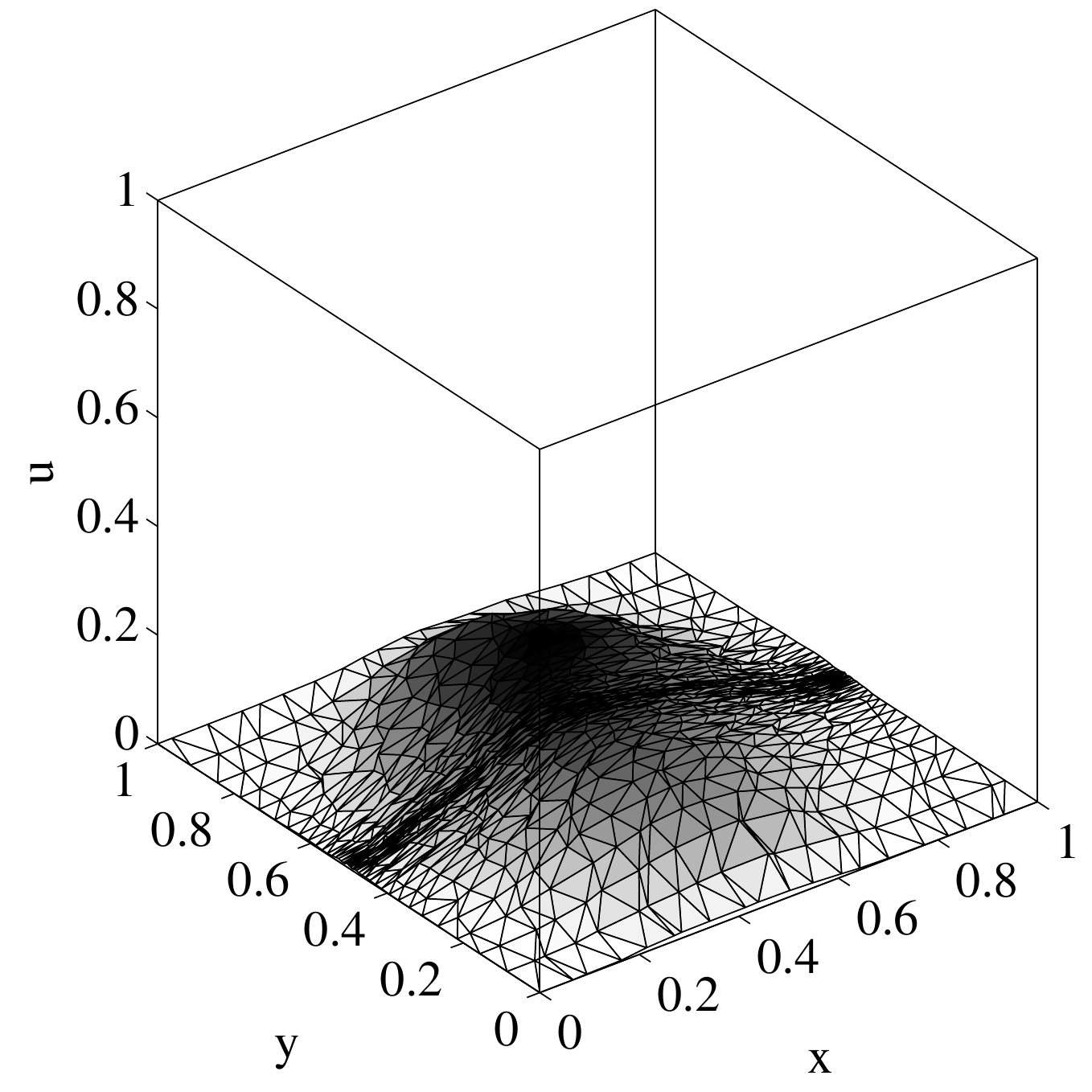}}\\
  \addtocounter{subfigure}{-3}
  \subfloat[ILR]
  {\includegraphics[width=.3\textwidth]{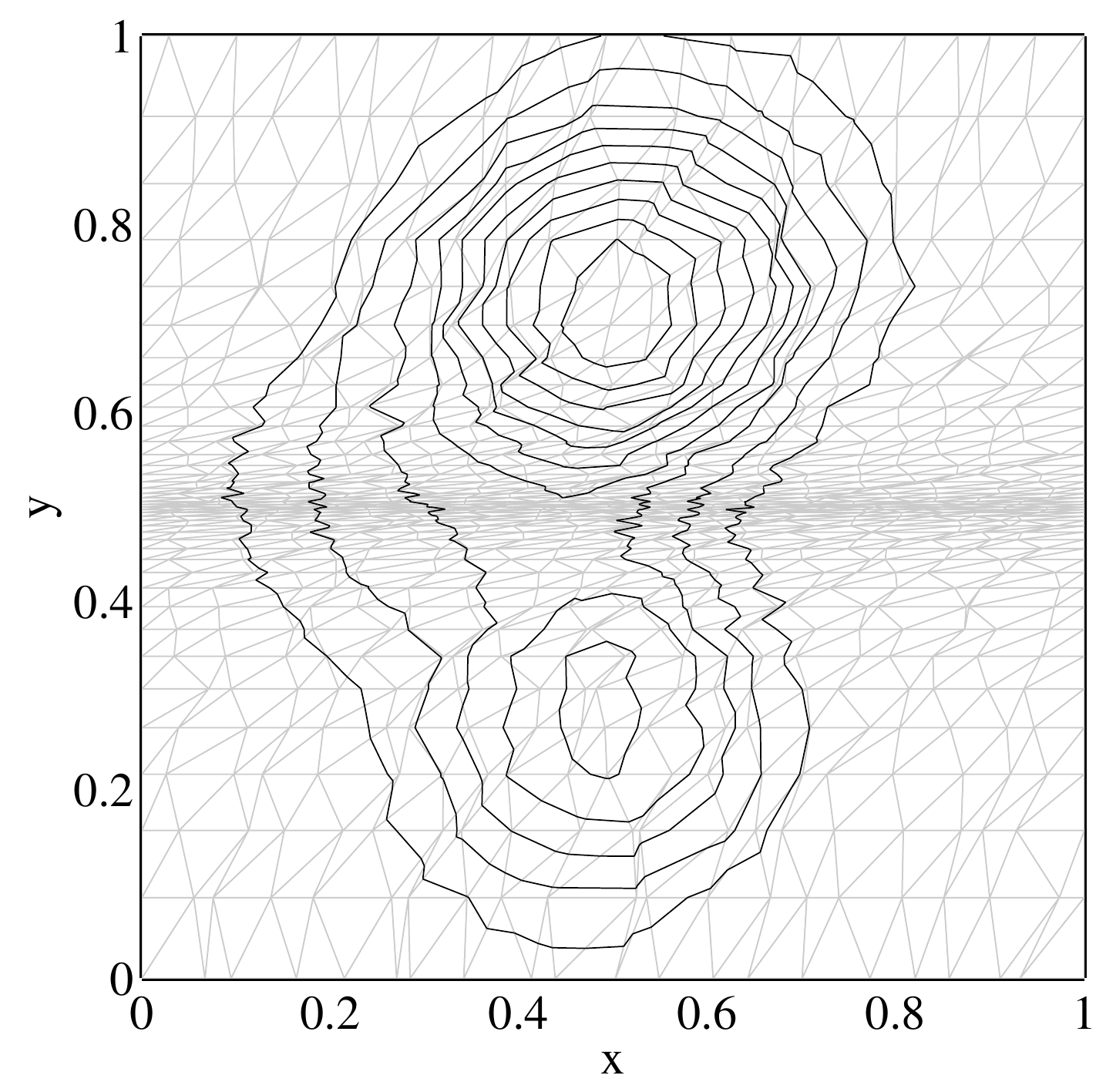}}
  \subfloat[MLP]
  {\includegraphics[width=.3\textwidth]{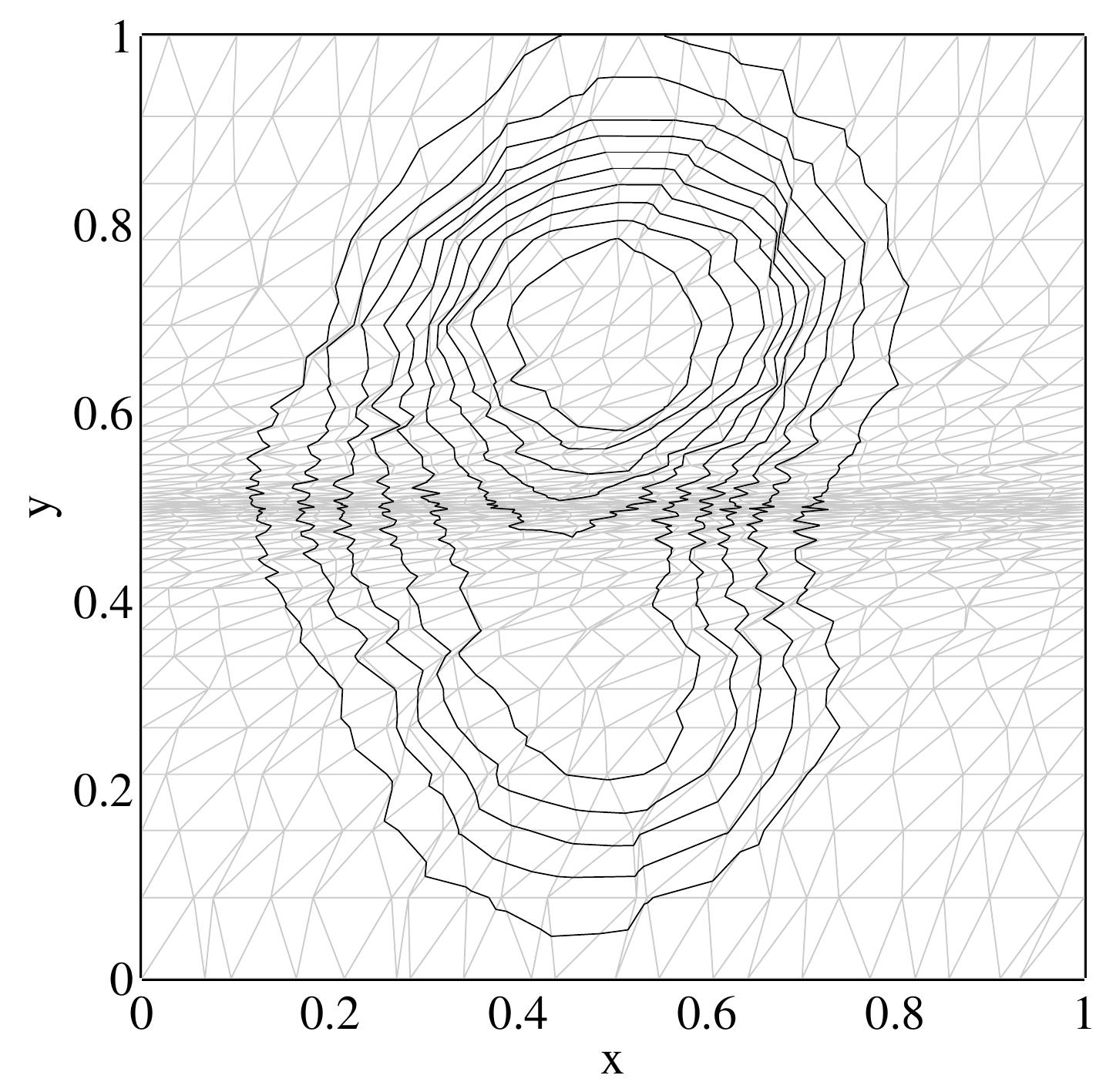}}
  \subfloat[Barth's limiter]
  {\includegraphics[width=.3\textwidth]{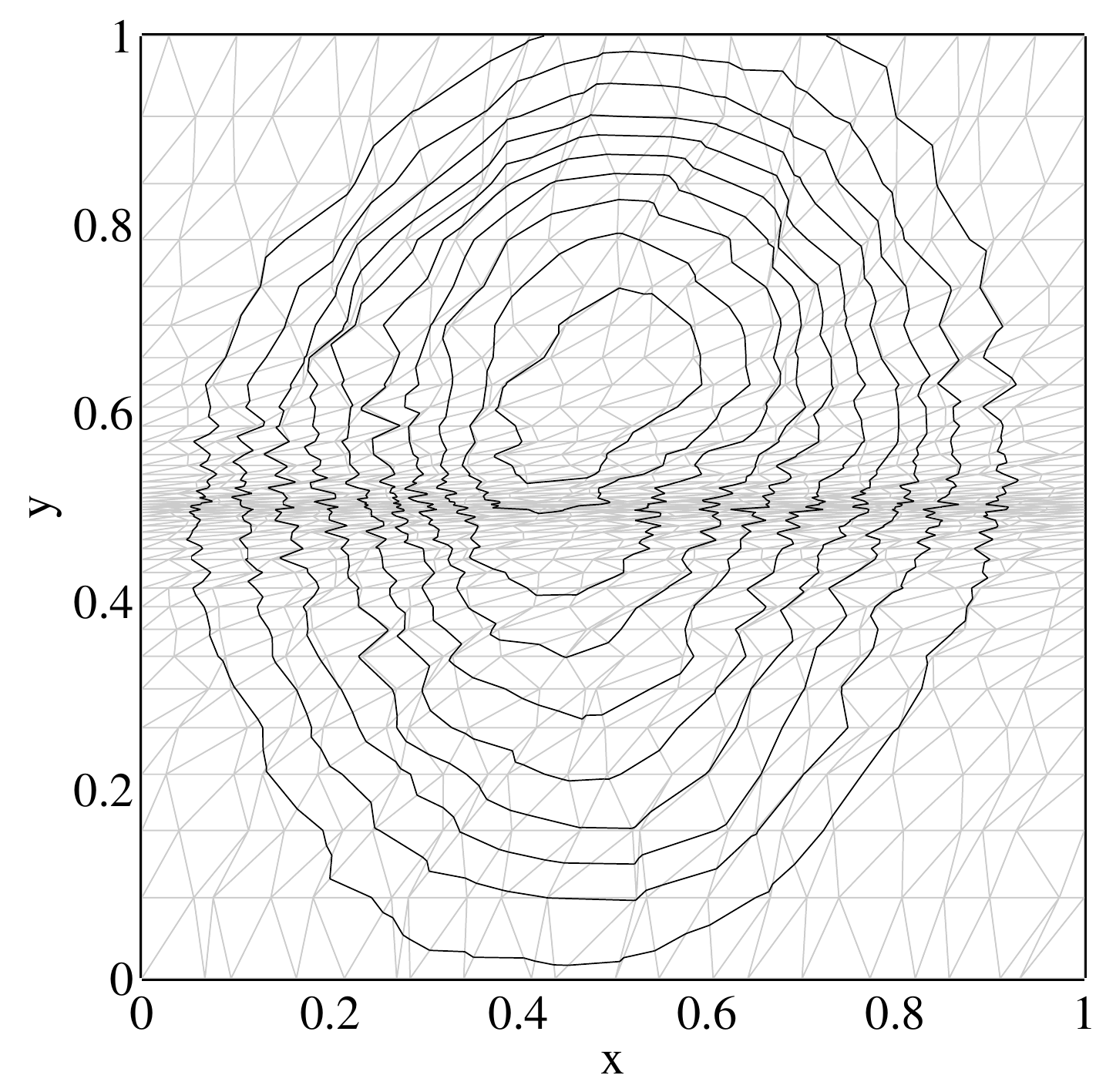}}
  \caption{Solutions of the solid body rotation problem at $t=2\pi$.}
  \label{fig:sbrcont}
\end{figure}

\subsection{Double Mach reflection}

As one benchmark of the Euler equations, the double Mach reflection 
problem \cite{Woodward1984} is considered here. The whole computational 
domain is $[0,4]\times[0,1]$. Initially, a right-moving Mach 10 shock is 
located at the beginning of the wall $(x,y)=(1/6,0)$, making a $60^\circ$ 
angle with the $x$-axis. The boundary setup is the same to \cite{Park2010}.
The HLLC flux is used as the numerical flux and computation is carried out 
until $t=0.2$.

Fig. \ref{fig:DM} shows the comparison of density contours computed on 
two successively refined Delaunay meshes. Obviously ILR provides a more 
stable structure for the Mach stem and slip line than Barth's limiter does. 
Fig. \ref{fig:DMcloseup} shows the close-up view around the Mach stem. 
Again ILR provides better resolution below the Mach stem, though the 
Barth's limiter captures the shear layer instability from the shock triple point 
more accurately.

\begin{figure}[htbp]
	\centering
	\subfloat[Barth's limiter: 23367 cells]
	{\includegraphics[width=.48\textwidth]{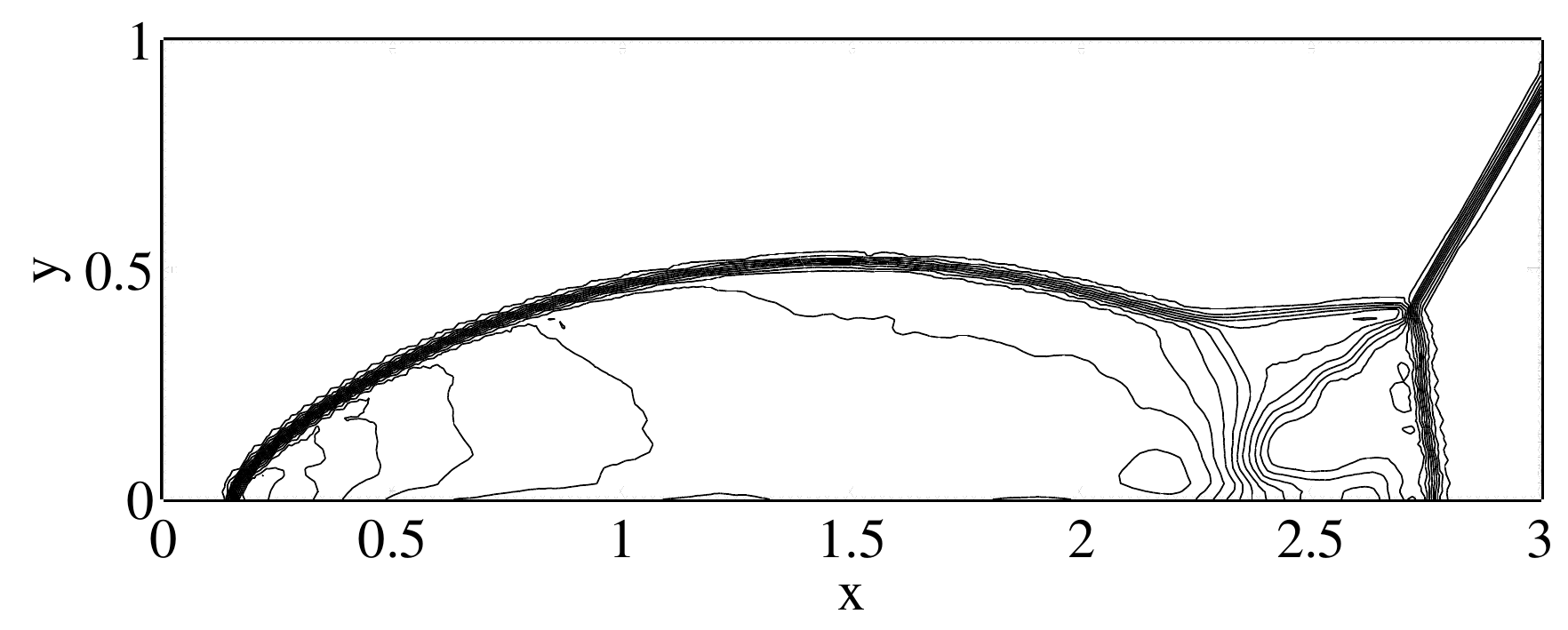}}
	\quad
	\subfloat[Barth's limiter: 93468 cells]
	{\includegraphics[width=.48\textwidth]{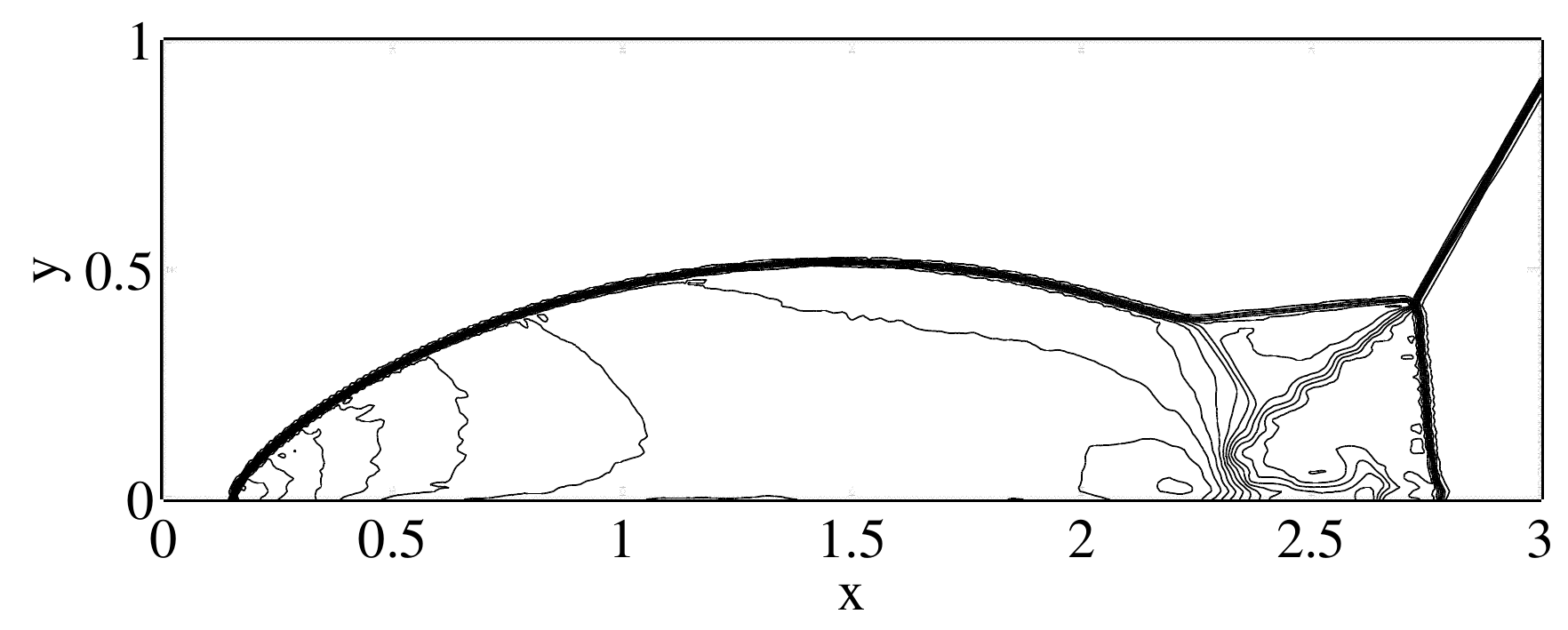}}\\
	\subfloat[ILR: 23367 cells]
	{\includegraphics[width=.48\textwidth]{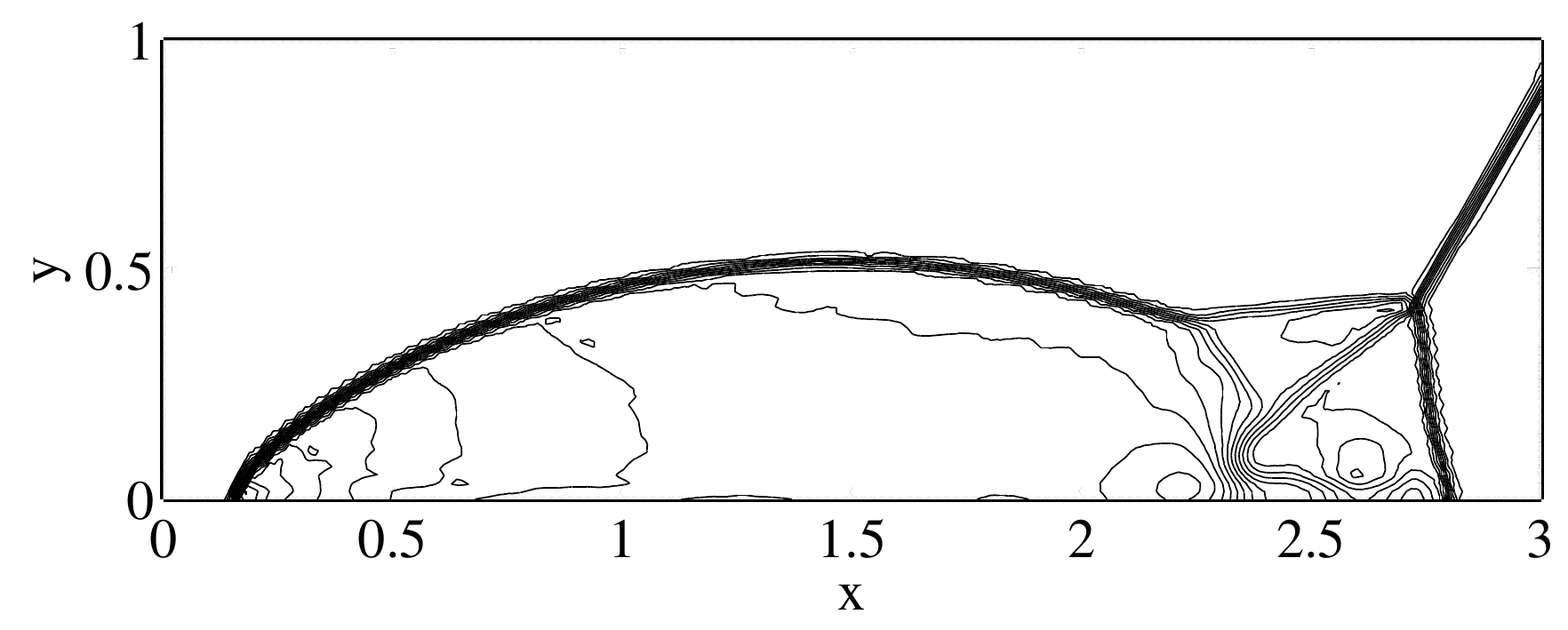}}
	\quad
	\subfloat[ILR: 93468 cells]
	{\includegraphics[width=.48\textwidth]{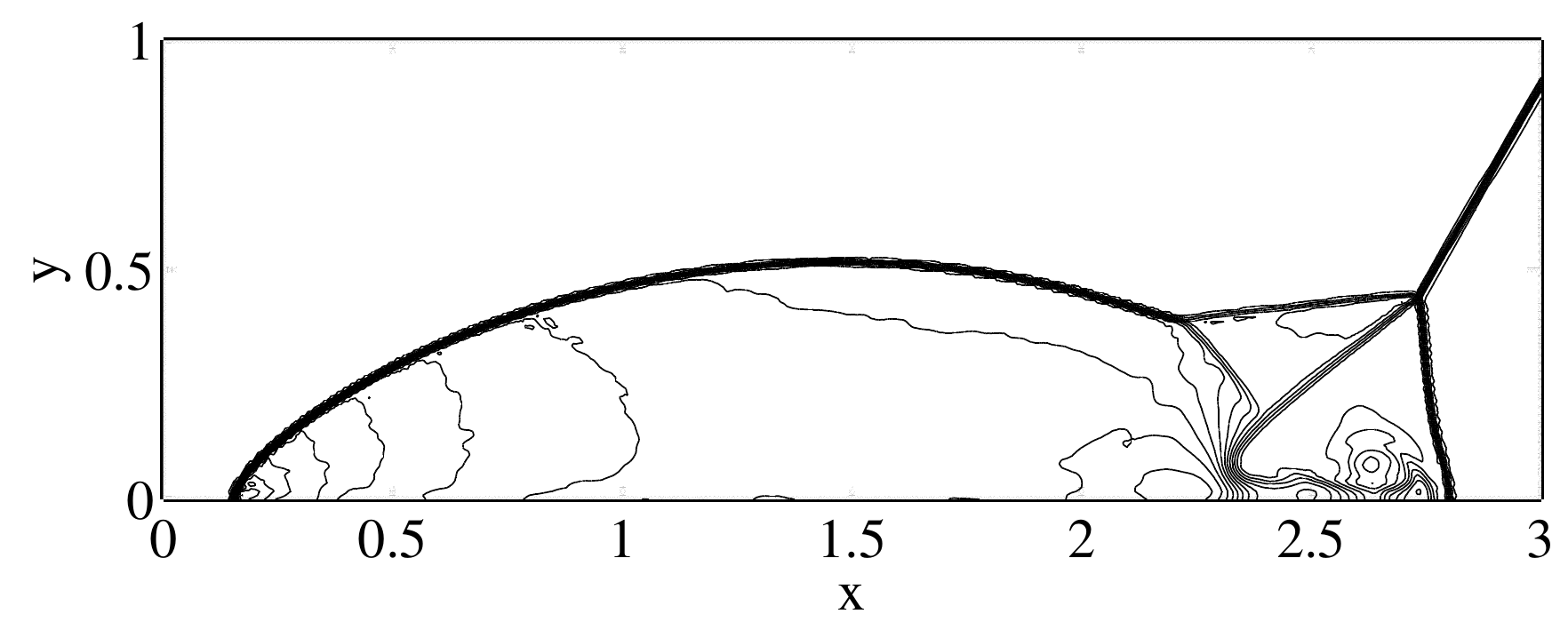}}
	\caption{Comparison of density contours for Double Mach reflection.
		         Thirty equally spaced contour lines from $\rho=1.5$ to 
		         $\rho=23.5$.}
	\label{fig:DM}
	\subfloat[Barth's limiter: 23367 cells]
	{\includegraphics[width=.45\textwidth]{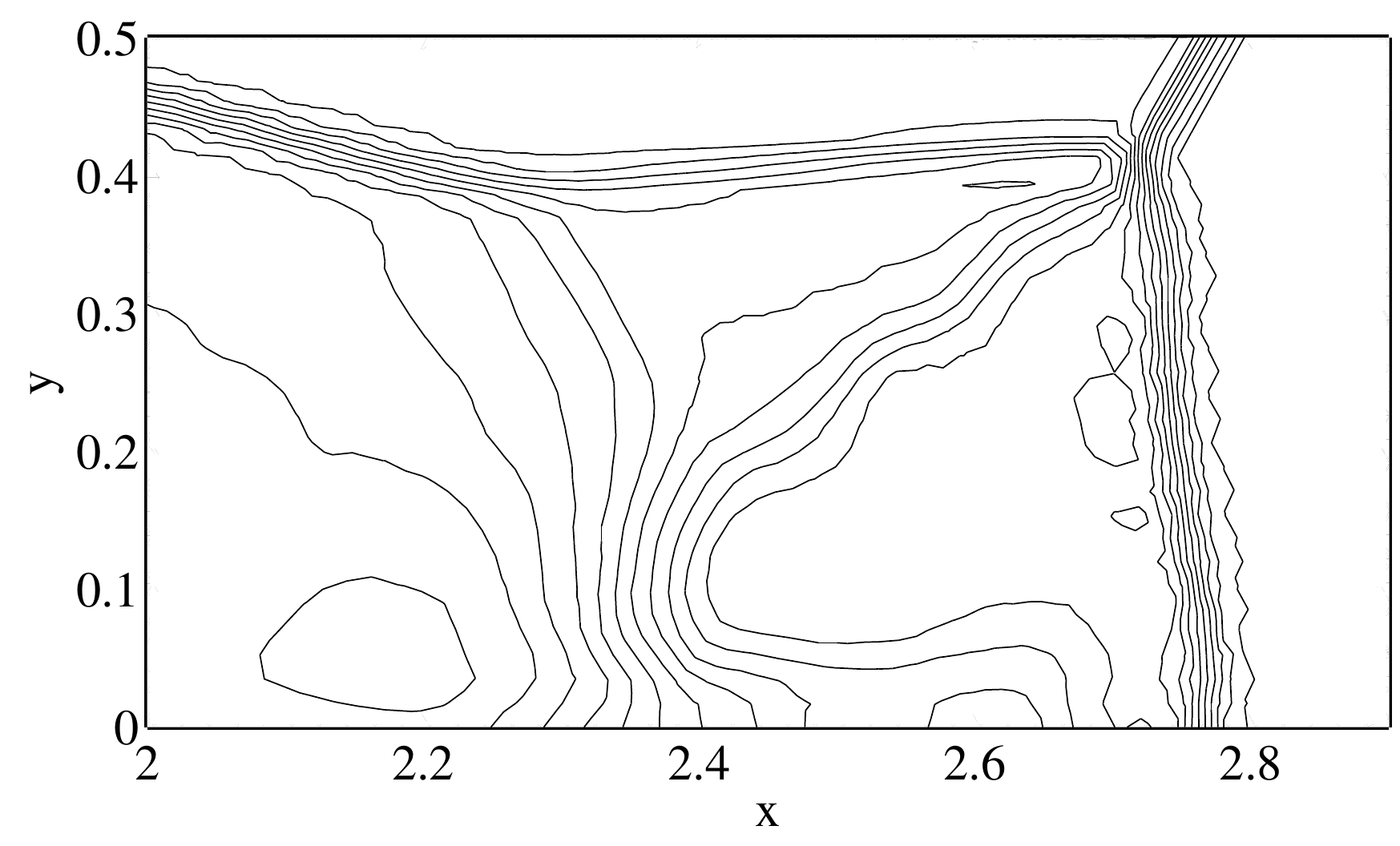}}
	\quad
	\subfloat[Barth's limiter: 93468 cells]
	{\includegraphics[width=.45\textwidth]{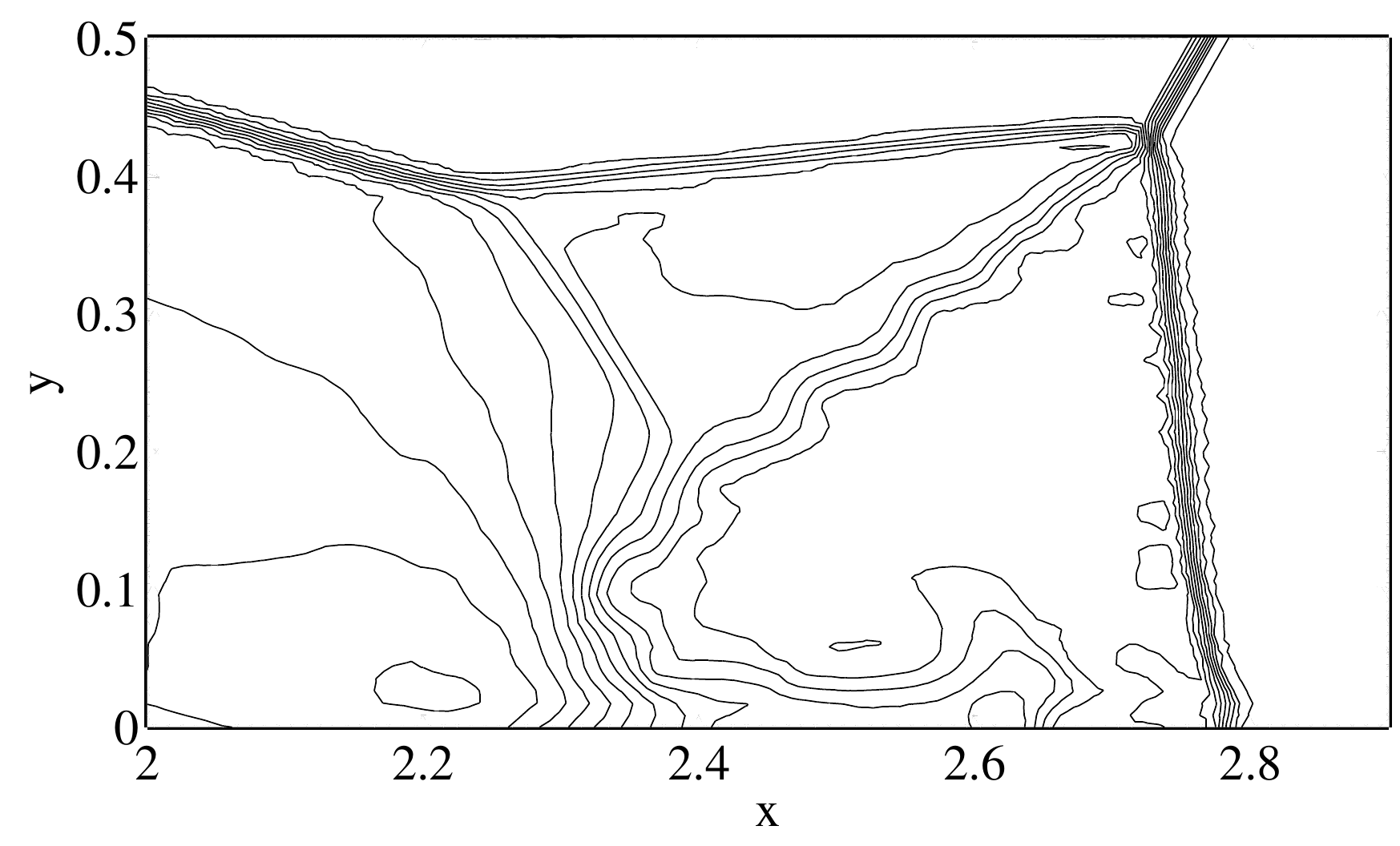}}\\
	\subfloat[ILR: 23367 cells]
	{\includegraphics[width=.45\textwidth]{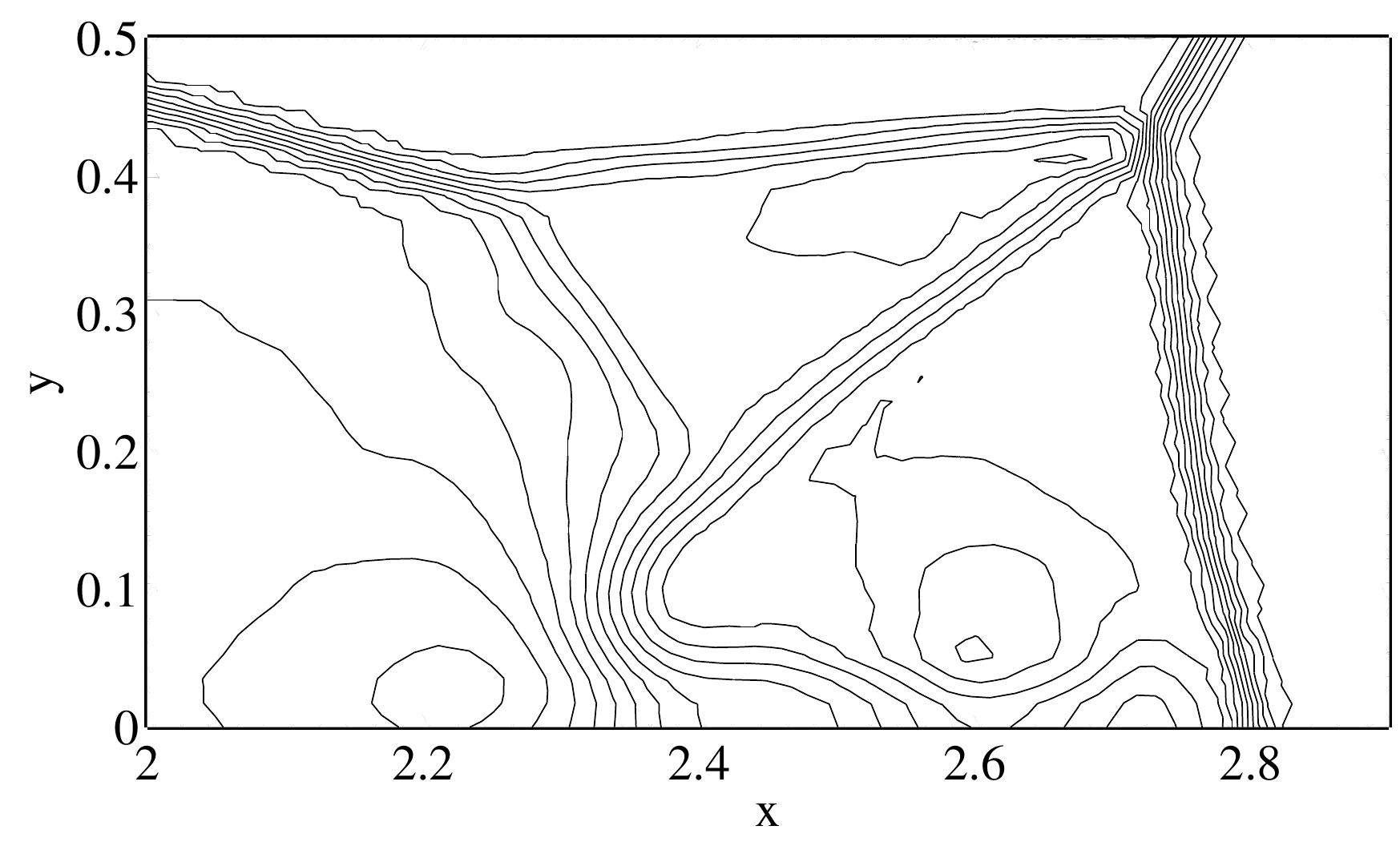}}
	\quad
	\subfloat[ILR: 93468 cells]
	{\includegraphics[width=.45\textwidth]{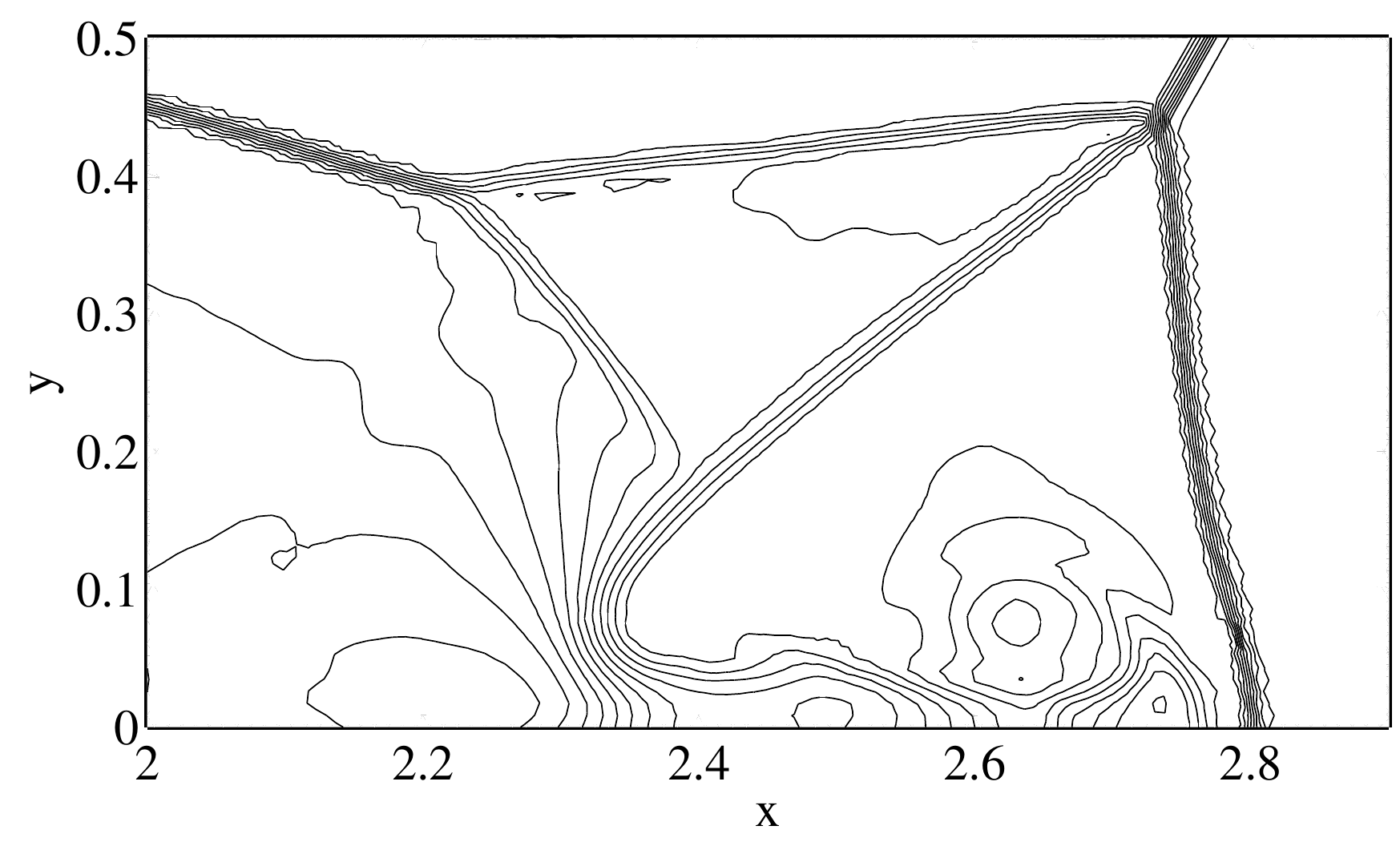}}
	\caption{Close-up view around the double Mach stem.}
	\label{fig:DMcloseup}
\end{figure}

\subsection{A Mach 3 wind tunnel with a step}

This is another popular test for high-resolution schemes 
\cite{Woodward1984}. The problem begins with a uniform Mach 3 flow in a 
wind tunnel with a step of 1 length unit high and 3 length units long. The 
step is $0.2$ length units high and is located $0.6$ length units from the 
left end of the tunnel. Reflective boundary condition is applied along the 
walls of the tunnel. And inflow and outflow boundary conditions are used at 
the left and right boundaries respectively. To resolve the singularity at the 
corner of the step, we refine the initial mesh near the corner, as is shown in 
Fig. \ref{fig:FWSinitmesh}. 

To capture the shock structure efficiently we apply the mesh adaptation. 
The adaptive indicator we use here is the jump of pressure. And the 
adaptation is carried out on each time level. The HLLC flux is used for this 
problem. Figs. \ref{fig:FWSmesh} and \ref{fig:FWScont} show the adaptive 
mesh and contour picture for the density at $t=4$. It can be observed that 
the ILR captures the discontinuities well and provides a high resolution of 
the slip line from the shock triple point.

\begin{figure}[htbp]
  \centering
  \includegraphics[height=.25\textheight]{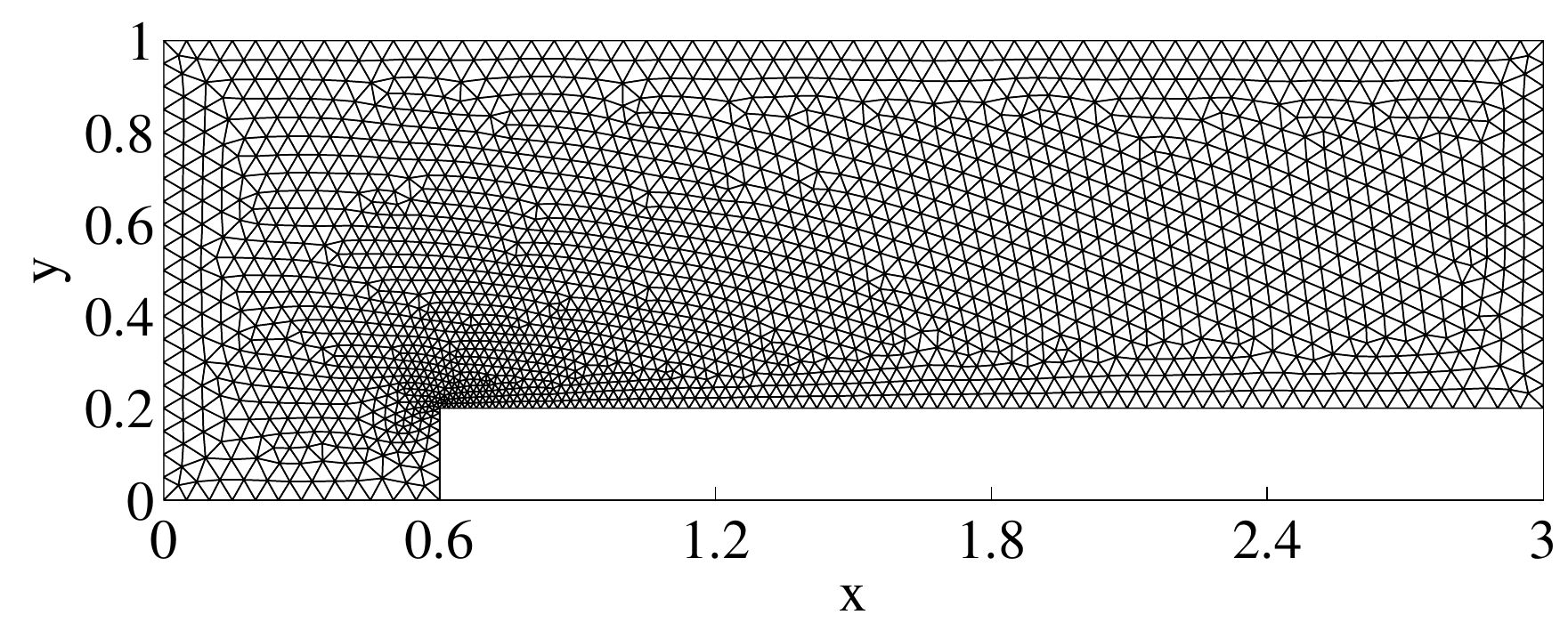}
  \caption{Initial background mesh with 4009 cells.}\label{fig:FWSinitmesh}
  \includegraphics[height=.25\textheight]{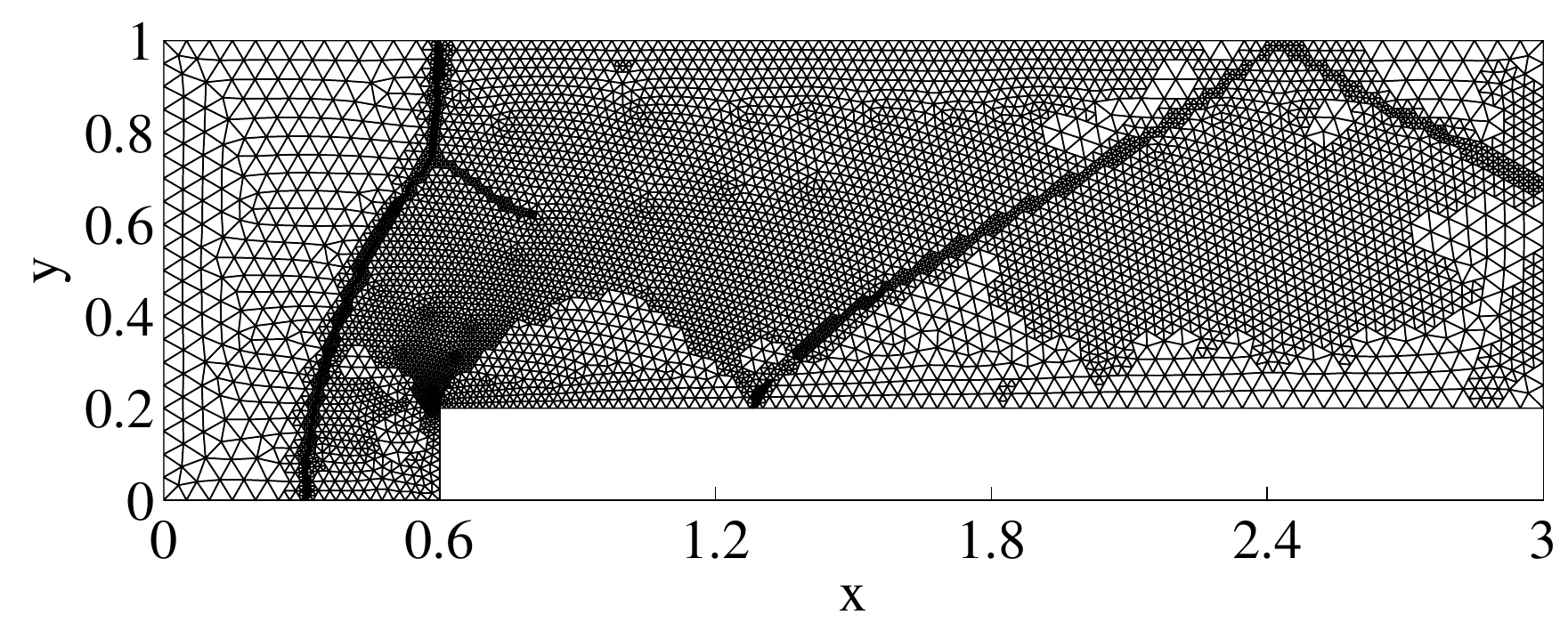}
  \caption{Final adaptive mesh with 13483 cells.}\label{fig:FWSmesh}
  \includegraphics[height=.25\textheight]{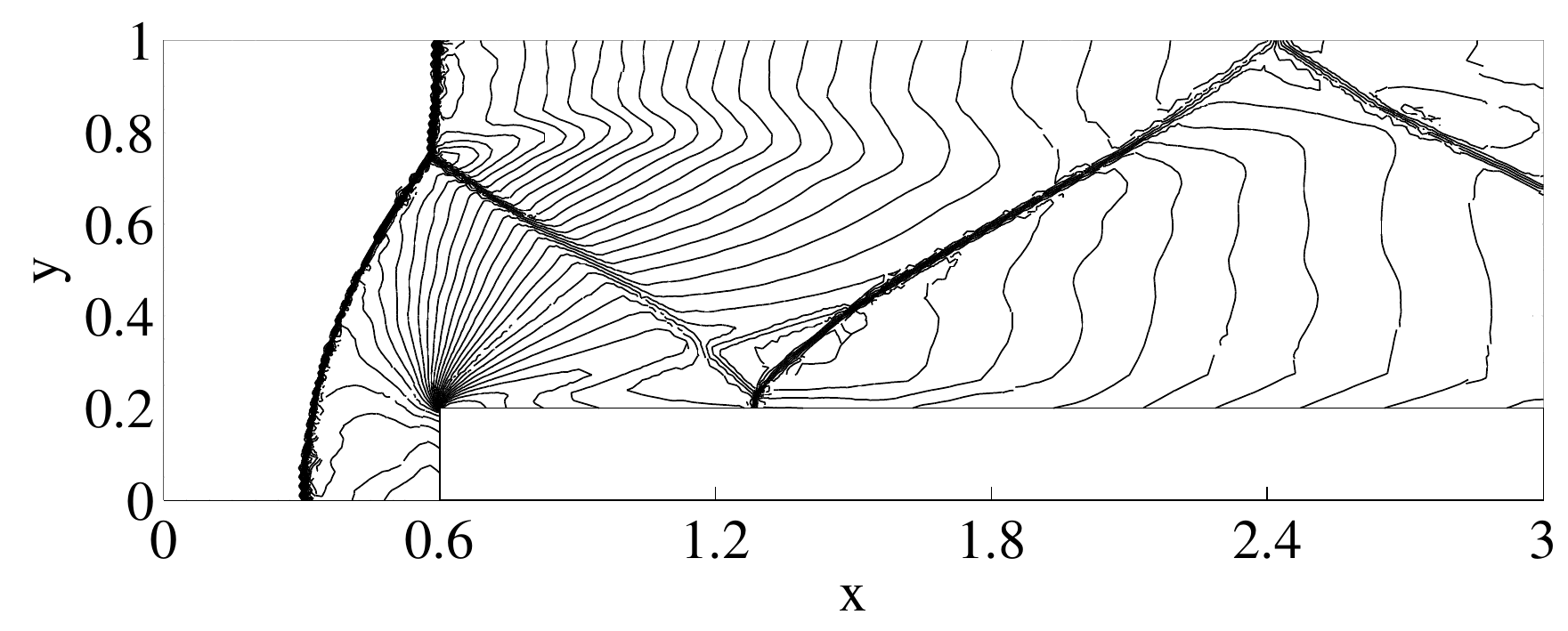}
  \caption{Thirty equally spaced contour lines from $\rho=0.32$ to        
  	           $\rho=6.15$.}\label{fig:FWScont}
\end{figure}

\subsection{Sedov point-blast wave}

The Sedov point-blast wave is a self-similar problem with both low density 
and low pressure. The analytical solution of this problem is available and 
can be found on \cite{Sedov1959,Kamm2000}. Without 
positivity-preserving treatment the computation may break down due to the 
presence of negative density or pressure. Here, we regard this spherically 
symmetric problem as three-dimensional cylindrically symmetric flow. The 
governing equation is 
\[
\dfrac{\partial r\bm u}{\partial t}+\dfrac{\partial r\bm f(\bm u)}{\partial r}
+\dfrac{\partial r\bm g(\bm u)}{\partial z}=\bm s(r\bm u,(r,z)),
\]
where
\[
\bm u=
\begin{pmatrix}
\rho \\ \rho u \\ \rho v \\ E
\end{pmatrix},\quad
\bm f(\bm u)=
\begin{pmatrix}
\rho u \\\rho u^2+p\\\rho uv  \\u(E+p)   
\end{pmatrix},\quad
\bm g(\bm u)=
\begin{pmatrix}
\rho v \\\rho uv \\\rho v^2+p \\v(E+p)
\end{pmatrix},\quad 
\bm s(r\bm u,(r,z))=
\begin{pmatrix}
0\\p\\0 \\0
\end{pmatrix},
\]
with 
\[
p=(\gamma-1)\left(E-\dfrac{1}{2}\rho(u^2+v^2)\right).
\]

Let $\bm U=r\bm u$, then formally we obtain the Euler equations with a
source term
\begin{equation}
\dfrac{\partial \bm U}{\partial t}+\dfrac{\partial \bm f(\bm U)}{\partial r}+
\dfrac{\partial \bm g(\bm U)}{\partial z}=\bm s(\bm U,(r,z)).
\label{eq:hclsrource}
\end{equation}
The finite volume discretization of \eqref{eq:hclsrource} we apply is
\begin{equation}
\bm U_0^{n+1}=\bm U_0^n-\dfrac{\Delta t_n}{|T_0|}\sum_{j=1}^J
\mathcal F(\bm U_j^-,\bm U_j^+;\bm n_j)|e_j|
+\sum_{j=1}^J \Delta t_nw_j\bm s(\bm U_j^-,\bm z_j),
\label{eq:fvmsource}
\end{equation}
where $(w_1,w_2,w_3)=(1/3,1/3,1/3)$ for $J=3$, and
$(w_1,w_2,w_3,w_4)=(1/3,1/3,1/6,1/6)$ for $J=4$. It can be proved that 
the scheme \eqref{eq:fvmsource} also preserves positivity provided that the 
time step length satisfies $a\Delta t_n \le \beta h/4$. See Appendix B for 
details of the proof. 

In the blast wave setup, a finite amount of energy $E_0=0.851072$ is 
released at the origin at $t=0$. The initial density $\rho_0=1$. These 
values are chosen so that the shock arrives at $r=1$ at the chosen time 
$t = 1$ \cite{Kamm2000}. A triangular grid on $[0,1.2]\times[0,1.2]$ with 
uniform spacing $l=1.2/80$ is used for computation. Symmetric boundary 
conditions are imposed at the bottom and left edges. The total energy $E$ 
is a constant $E_0/2\pi l^3$ in the two cells at the left-bottom corner and 
$10^{-12}$ elsewhere (emulating a $\delta$-function at the origin). Again 
we use the HLLC flux as numerical flux. The contour lines of density and the 
profile along the radial direction are displayed in Fig. \ref{fig:sedov}, which 
demonstrate the robustness and stability of ILR.
\begin{figure}[htbp]
	\centering 
	\subfloat[Twenty equally spaced contour lines 
	from $\rho=0$ to $\rho = 6$.]
	{\includegraphics[width=.38\textwidth]{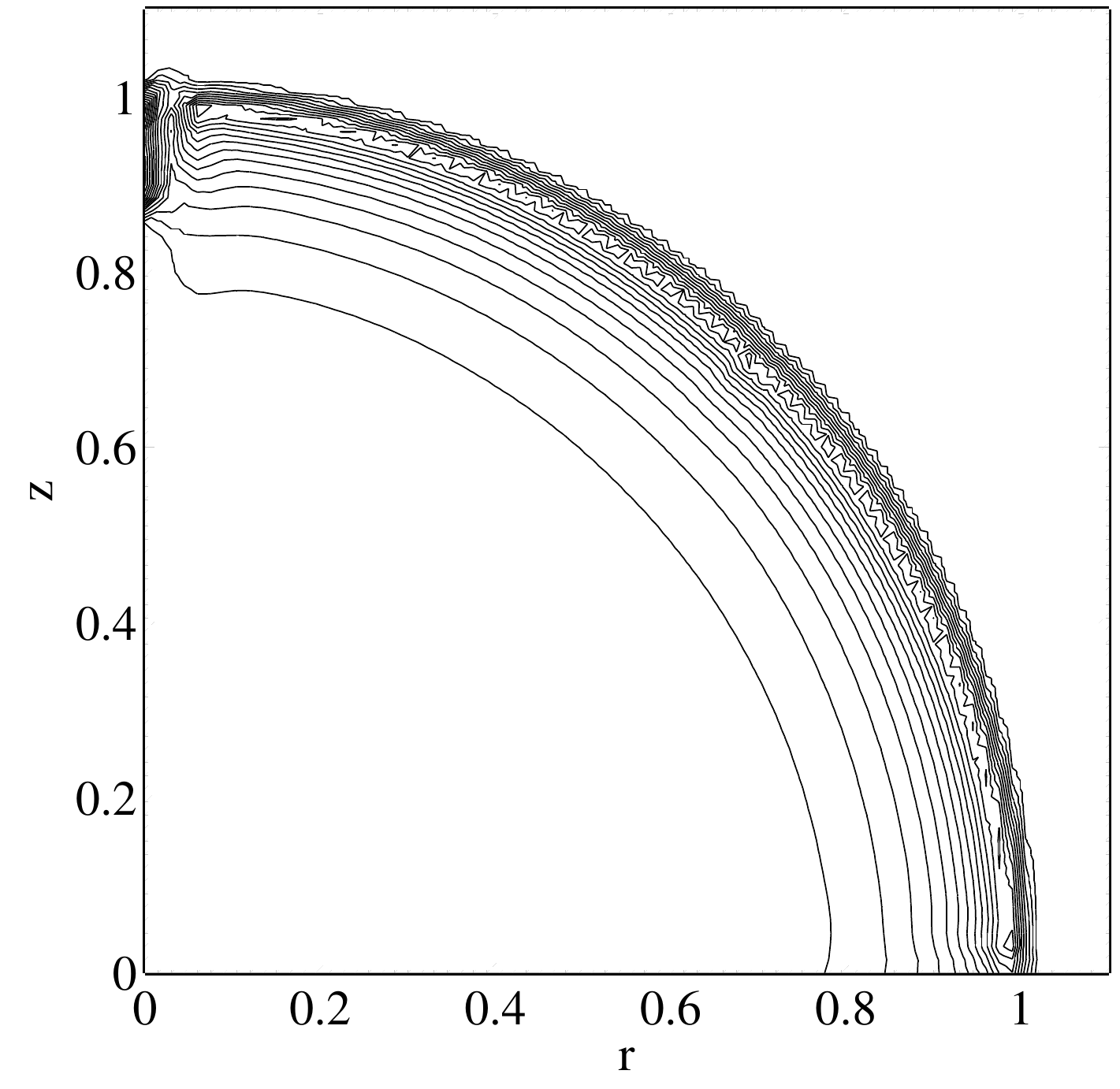}}
	\quad
	\subfloat[Projection to radial coordinates.
	The solid line represents the exact solution
	and the dots represent the numerical solution.]
    {\includegraphics[width=.46\textwidth]{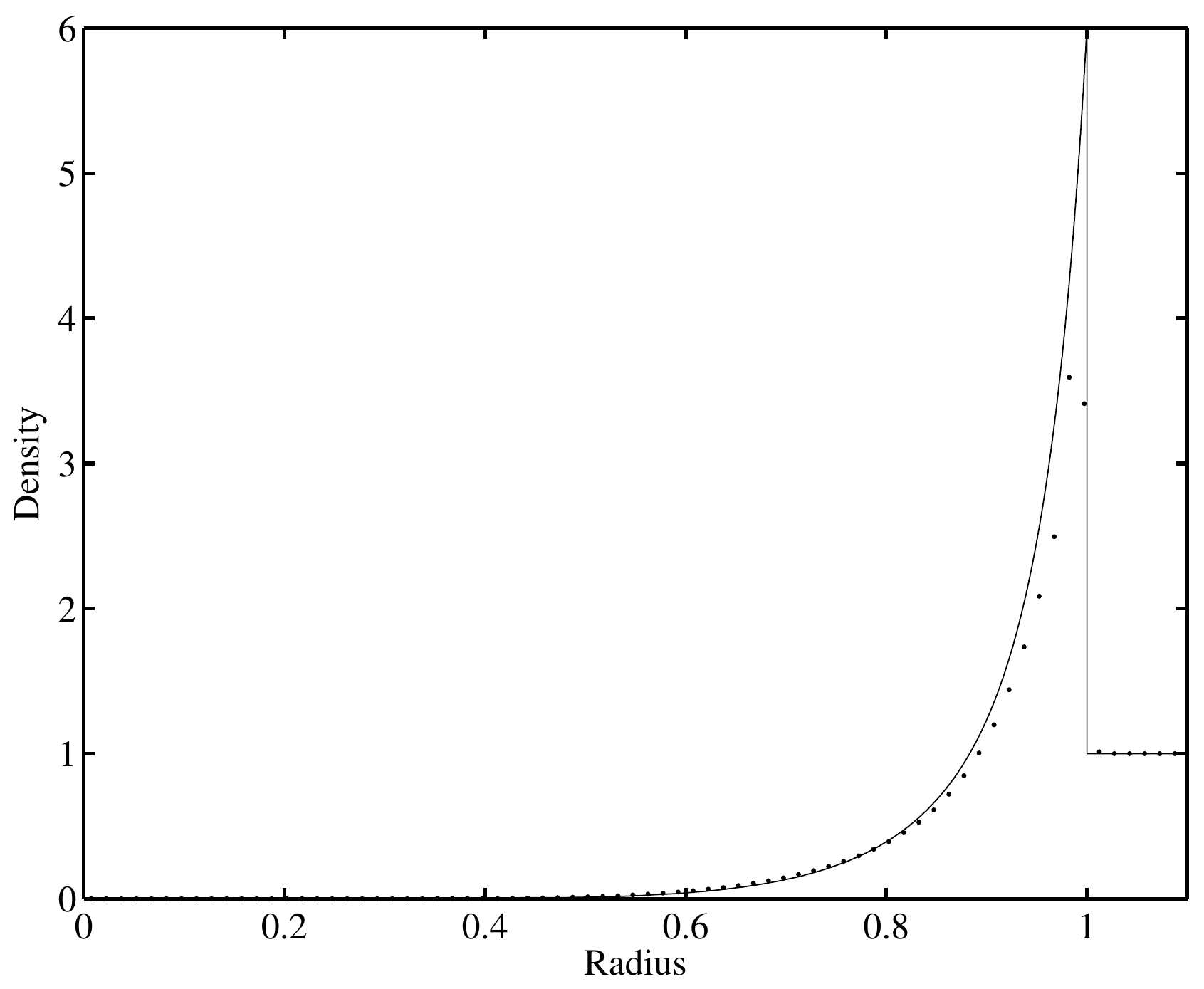}}
    \caption{Sedov blast wave.  
    Result at $t=1$ on an $80\times 80\times 2$ mesh.}
    \label{fig:sedov}
\end{figure}

\section{Conclusion}

We present an integrated linear reconstruction for finite volume scheme on 
arbitrary unstructured grids.  We use the sum of squared residuals as an 
objective function, and the inequalities on midpoints as linear constraints.  
The resulting optimization problem is a convex quadratic programming 
problem which can be efficiently solved using the active-set methods.  
Numerical tests are provided to show the capacity of this new technique to 
deal with computations on a variety of unstructured meshes.  Extension to 
higher-order reconstruction is possible, by increasing the number of 
variables and constraints in the quadratic programming problems.

\section*{Acknowledgements}
The first and third authors gratefully acknowledge the support of the
National Natural Science Foundation of China (Grant No. 91330205, 
11421110001, 11421101 and 11325102). 
The second author would like to thank the support of FDCT 050/2014/A1 
from Macao SAR, MYRG2014-00109-FST from University of Macau, and the
National Natural Science Foundation of China (Grant No. 11401608).

\begin{appendices}
\section{Procedures of active-set methods}

The procedures of active-set method for the double-inequality constrained 
quadratic programming \eqref{eq:dic} are listed in the following Algorithm 1.
\begin{algorithm*}[!h]
  \caption{\em Active-set methods for 
  	          double-inequality constrained quadratic programming}
  \begin{algorithmic}
  	\small
    \STATE $\bm L\leftarrow\bm 0$;
    \STATE $(\delta_1,\delta_2,\cdots,\delta_J)\leftarrow(0,0,\cdots,0)$;
    \WHILE {\TRUE}
        \STATE $\mathbf M\leftarrow[\delta_j\bm a_j^\top]_
                {\delta_j\neq 0}$
        \STATE 
        $\bm\lambda=[\lambda_j]_{\delta_j\neq 0}\leftarrow
        (\mathbf M\mathbf G^{-1}\mathbf M^\top)^{-1}
        \mathbf M(\bm L+\mathbf G^{-1}\bm c)$
        \STATE 
        $\bm p\leftarrow
        \mathbf G^{-1}(\mathbf M^\top\bm{\lambda}-\bm c)
        -\bm L$
        \IF {$\bm p=\bm 0$}
            \IF {$\mathbf M$ is empty \OR $\bm\lambda \ge \bm 0$ }
                 \STATE  \textbf{exit} with optimal solution $\bm L$
            \ELSE 
                \STATE 
                $j\leftarrow \arg\min\limits_{\delta_j\neq 0}\lambda_j$
                \STATE $\delta_j\leftarrow 0$
            \ENDIF 
         \ELSE 
          \STATE
          $\beta_j
          \leftarrow
          \begin{cases}
          \dfrac{M_j-u_0-\bm a_j^\top\bm L}{\bm a_j^\top\bm p},
          & \bm a_j^\top\bm p > 0, \\
          \dfrac{m_j-u_0-\bm a_j^\top\bm L}{\bm a_j^\top\bm p},
          & \bm a_j^\top\bm p < 0, \\
          +\infty, & \bm a_j^\top\bm p=0.
          \end{cases}$
           \STATE 
           $\alpha\leftarrow\min\left\{1,\min\limits_{1\le j\le J}
           \beta_j \right\}$
           \STATE $\bm L\leftarrow\bm L+\alpha\bm p$
              \IF {$\alpha=\beta_k$}
              \STATE $\delta_k\leftarrow\ -\sgn(\bm a_k^\top\bm p)$.
              \ENDIF    
        \ENDIF 
    \ENDWHILE
  \end{algorithmic}
\end{algorithm*}

There are two aspects regarding the numerical issue in our actual 
implementation.  Firstly, although the matrix $\mathbf M$ in the active-set 
method is theoretically full rank at each iteration, we cannot guarantee it 
numerically.  Therefore, the condition $\bm a_j^\top\bm p\gtrless 0$ in 
the computation of the coefficients $\beta_j$ is relaxed to 
$\bm a_j^\top\bm p\gtrless \epsilon U$ accordingly, where 
$\epsilon=10^{-12}$ and $U$ denotes the reference magnitude of the 
solution.  Similarly, the termination condition of the Lagrange multipliers is 
relaxed to $\bm\lambda\ge \epsilon U$. Secondly, the maximum number 
of iterations is chosen as $6$, which is adequate for most two-dimensional 
numerical tests. 

\section{Positivity of axisymmetric Euler equations}

The positivity of axisymmetric Euler equations can be verified using the 
technique in \cite{Zhang2011}. For simplicity, we only consider the purely 
triangular grid. The finite volume scheme \eqref{eq:fvmsource} can be 
rewritten as 
\[
\bm U_0^{n+1}=\bm U_0^{n}
-\left(\lambda(|e_1|+|e_2|+|e_3|)+\dfrac{1}{2}\beta\right)
(\bm U_1^-+\bm U_2^-+\bm U_3^-)+
\dfrac{1}{2}\beta\sum_{j=1}^3 \left(\bm U_j^-+\dfrac{2}{3\beta}
\Delta t_n\bm s(\bm U_j^-,\bm z_j)\right)+\tilde {\bm U},
\]
where $\tilde {\bm U}\in\mathcal G$ is the sum of last nine terms in 
\eqref{eq:ppsplit}, with $\bm U_1^-,\bm U_2^-$ and $\bm U_3^-$ in 
place of $\bm u_1^-,\bm u_2^-$ and $\bm u_3^-$ respectively. 
Therefore, to guarantee the positivity of $\bm U_0^{n+1}$, the time step 
length $\Delta t_n$ satisfies 
\[
a\Delta t_n\le \dfrac{1}{4}\beta h \text{~and~}\bm U_j^-+
\dfrac{2}{3\beta}\Delta t_n\bm s(\bm U_j^-,\bm z_j)\in\mathcal G.
\]

In the following derivation we drop the superscripts and subscripts. Let 
$\tau=2\Delta t_n/3\beta$. Obviously 
$\rho(\bm U+\tau\bm s(\bm U,\bm z))=r\rho>0$. 
A direct algebraic calculation shows that 
\[
\begin{split}
p(\bm U+\tau\bm s(\bm U,\bm z))
&=(\gamma-1)\left(rE-\dfrac{(r\rho u+\tau p)^2+(r\rho v)^2}
{2r\rho}\right)\\
&=\left(r-(\gamma-1)u\tau-\dfrac{(\gamma-1)c^2}{2\gamma r}
\tau^2\right)p,
\end{split}
\]
and as a result,
\[
\begin{split}
p(\bm U+\tau\bm s(\bm U,\bm z))>0 & \Longleftrightarrow
r-(\gamma-1)u\tau-\dfrac{(\gamma-1)c^2}{2\gamma r}\tau^2 > 0\\
&\Longleftrightarrow
\tau<\tau_0:=\left(\left(\dfrac{\gamma ru}{c^2}\right)^2+\dfrac
{2\gamma r^2}{(\gamma-1)c^2}\right)^{1/2}-\dfrac{\gamma ru}{c^2}.
\end{split}
\]
Since
\[
\begin{split}
\tau_0  = &~ \dfrac{4\gamma r}{\left((\gamma(\gamma-1)u)^2
	+2\gamma(\gamma-1)c^2\right)^{1/2}+\gamma(\gamma-1)u}\\
            > &~ \dfrac{4\gamma r}{2\gamma(\gamma-1)|u|+
            	     \sqrt{2\gamma(\gamma-1)}c} \\
         \ge &~ \dfrac{4\gamma r}{\max\left\{2\gamma(\gamma-1),
         	\sqrt{2\gamma(\gamma-1)}\right\}a} \\
            > &~ \dfrac{4\gamma r}{2\gamma^2 a}=\dfrac{2r}{\gamma a},
\end{split}
\]
we obtain the following CFL-like condition
\[
a\Delta t_n \le \beta\min\left\{\dfrac{h}{4},
\dfrac{3r_{\min}}{\gamma}\right\},
\]
where $r_{\min}$ is the minimum radius of quadrature points under 
consideration.
\end{appendices}

\end{document}